\documentclass[12pt,centertags,oneside,reqno]{amsart}
\usepackage{amsmath,amsthm,amssymb,bm}

\usepackage{amscd,amsxtra,calc}
\usepackage{cmmib57}
\usepackage{url}

\usepackage{amscd}
\usepackage{pstricks}
\usepackage{color}
\usepackage[a4paper,width=16.2cm,top=3cm,bottom=3cm]{geometry}

\numberwithin{equation}{section}

\theoremstyle{plain}

\newtheorem{thm}{Theorem}[section]
\newtheorem{theorem}[thm]{Theorem}
\newtheorem*{thma}{Theorem A}
\newtheorem*{corb}{Corollary B}
\newtheorem*{thmc}{Theorem C}
\newtheorem*{thmc'}{Theorem C'}
\newtheorem*{thmd}{Theorem D}

\newtheorem{lemma}[thm]{Lemma}
\newtheorem{corollary}[thm]{Corollary}
\newtheorem{proposition}[thm]{Proposition}

\theoremstyle{definition}

\newtheorem{remark}[thm]{Remark}

\newtheorem{definition}[thm]{Definition}

\newtheorem{example}[thm]{Example}
\newtheorem{examples}[thm]{Examples}

\newtheorem{question}[thm]{Question}

\numberwithin{equation}{section}

\newtheorem{report}{report}
\newtheorem{exercise}{excercise}[section]

\newenvironment{rem}{\begin{remark}\rm}{\end{remark}}

\renewcommand{\labelenumi}{\rm{(}\arabic{enumi}\rm{)}}

\usepackage{mleftright}
\usepackage{mdwlist}

\usepackage{etoolbox}

\makeatletter
\def\subsection{\@startsection{subsection}{1}%
  \z@{.5\linespacing\@plus.7\linespacing}{-.5em}%
  {\normalfont\itshape}}
\patchcmd{\@sect}
  {\@addpunct.}
  {}
  {}
  {}
\makeatother

\makeatletter
\def\subsection{\@startsection{subsection}{2}%
  \z@{.5\linespacing\@plus.7\linespacing}{.3\linespacing}%
  {\normalfont\bfseries}}
\makeatother

\usepackage{caption}

\usepackage{amsmath}

\usepackage{mathtools}
\DeclarePairedDelimiter{\abs}{\lvert}{\rvert}

\newcommand{\id}{\mathrm{id}}

\usepackage{graphicx}

\usepackage{comment}
\title[]{Dynamical degrees of automorphisms of complex simple abelian varieties and Salem numbers}

\author{Yutaro Sugimoto
}
\begin{document}
\maketitle

\begin{abstract}
 We prove that every Salem number can be realized as the first dynamical degree of an automorphism of a complex simple abelian variety.
 Also by using the similar technique, we prove that the set of first dynamical degrees of automorphisms of complex simple abelian varieties except $1$ has the minimum value when fixing the dimension of complex simple abelian varieties.
 Moreover, we prove that there is an automorphism of a complex simple abelian variety, whose first dynamical degree is as close as possible to $1$.
 These results are inspired by the work of Nguyen-Bac Dang and Thorsten Herrig.
\end{abstract}

\section{Introduction}\label{Introduction}
 For a $g$-dimensional compact K\"{a}hler manifold $X$ and a dominant meromorphic map $f\colon X\dashrightarrow X$, the $k$-th dynamical degree $\lambda_k(f)$ of $f$ is defined as 
 \begin{align*}
  \lambda_k(f):=\lim_{n\to+\infty}||(f^n)^{*}:\mathrm{H}^{k,k}(X)\rightarrow\mathrm{H}^{k,k}(X)||^{\frac{1}{n}}
 \end{align*}
 where $||\cdot||$ is a norm for linear transformations and $\mathrm{H}^{k,k}(X)$ ($0\leq k\leq g$) is the $(k,k)$-Dolbeault cohomology.
 The dynamical degrees are always not less than $1$ (see e.g. \cite{DS05} or \cite{DS17}).\par
 As mentioned in Section \ref{Dynamical degree}, if $f$ is a holomorphic map, $\lambda_k(f)=\rho(f^{*})$ where $\rho$ is the spectral radius and $f^*$ is on $\mathrm{H}^{k,k}(X)$.\par
 There are some results about the relevance between the first dynamical degrees and Salem numbers as well as Pisot numbers (cf.\ Section \ref{Algebraic preparation}).\par
 \begin{theorem}[{\cite[Theorem 5.1(1)]{DF01}}]
  Let $X$ be a compact K\"{a}hler surface and $f\colon X\dashrightarrow X$ be a bimeromorphic map with $\rho(f^{*})>1$ for the operator $f^{*}$ on $\mathrm{H}^{1,1}(X)$ where $\rho$ is the spectral radius.\par
  Then the operator $f^{*}$ has exactly one eigenvalue $\lambda\in\mathbb{R}_{>0}$, of modulus $|\lambda|>1$, and in fact $\lambda=\rho(f^{*})$.
 \end{theorem}
 \begin{remark}
  This theorem implies that the first dynamical degree of a birational map of a projective surface over the complex number field is $1$ or a Salem number or a Pisot number (cf.\ \cite[Theorem 1.2]{BC16}).
 \end{remark}
 \begin{theorem}[{\cite[Theorem 2.1 (i)]{DH22}}]\label{DH22 theorem}
  Let $\lambda$ be a Salem number of degree $g$.
  Then there exists an automorphism $f$ of a $g$-dimensional simple abelian variety $X$ with totally indefinite quaternion multiplication (i.e., the endomorphism algebra of $X$ is a totally indefinite quaternion algebra) with dynamical degrees
  \begin{align*}
   \lambda_1(f)=\cdots=\lambda_{g-1}(f)=\lambda^2 \quad \textit{and} \quad \lambda_0(f)=\lambda_g(f)=1.
  \end{align*}
 \end{theorem}
 This theorem concludes that there is an automorphism of some simple abelian variety which realizes the square of every Salem number, which is again a Salem number, as the first dynamical degree.\par
 The extension of Theorem \ref{DH22 theorem} is the main result of this paper as below.
 \begin{thma}
  Let $P(x)\in\mathbb{Z}[x]$ be an irreducible monic polynomial whose all roots are either real or of modulus $1$.
  Assume at least one root has modulus $1$.
  Let $g$ be the degree of the polynomial $P(x)\in\mathbb{Z}[x]$ and $z _1,z _2,\ldots,z_g$ be the roots of $P(x)$ which are ordered as $\abs{z_1}\geq\abs{z_2}\geq\cdots\geq\abs{z_g}$.\par
  Then, there is a simple abelian variety $X$ of dimension $g$ and an automorphism $f\colon X\longrightarrow X$ with dynamical degrees
  \begin{center}
   $\lambda_0(f)=1, \lambda_k(f)=\prod_{i=1}^{k} \abs{z_i}^2\,(1\leq k\leq g)$.
  \end{center}
 \end{thma}
 This theorem is proved in Section \ref{Main theorem} and it induces the next corollary in Section \ref{Corollaries}.
 \begin{corb}
  Any Salem number is realized as the first dynamical degree of an automorphism of a simple abelian variety over $\mathbb{C}$.
 \end{corb}
 Section \ref{Small first dynamical degree} is devoted to exploring small first dynamical degrees except $1$ and the next theorem is proved in Section \ref{Small dynamical degree with restricted dimension}.
 (Probably, this fact is famous and the proof is trivial, but we could not found the reference.)
 \begin{thmc}
  For an integer $g\geq2$,
  \begin{align*}
   \mathcal{A}_g:=\left\{
                   \begin{array}{l}
                    \text{first dynamical degrees of surjective endomorphisms}\\
                    \text{  of abelian varieties over $\mathbb{C}$ whose dimension is $g$}
                   \end{array}
                  \right\}\backslash\{1\},\\
   \mathcal{B}_g:=\left\{
                   \begin{array}{l}
                    \text{first dynamical degrees of automorphisms}\\
                    \text{  of simple abelian varieties over $\mathbb{C}$ whose dimension is $g$}
                   \end{array}
                  \right\}\backslash\{1\}
  \end{align*} 
  have the minimum value.
 \end{thmc}
 But, without restricting the dimension of (simple) abelian varieties, there does not exist the minimum value of the set of first dynamical degrees except $1$.
 This result is proved in Section \ref{First dynamical degree close to 1}.
 \begin{thmd}
  There does not exist the minimum value of 
  \begin{align*}
   \Delta:=\{\text{first dynamical degrees of automorphisms of simple abelian varieties over $\mathbb{C}$}\}\backslash\{1\}.
  \end{align*}
 \end{thmd}

 \medskip\noindent
 {\bf Acknowledgements.} The author thanks Professor Keiji Oguiso for suggesting some of the problem of dynamical degrees and commenting to this paper.
 He also thanks Long Wang for pointing out an error in the proof at the seminar.

\section{Preliminaries}
 \subsection{Dynamical degrees}\label{Dynamical degree}
  Let $X$ be a $g$-dimensional compact K\"{a}hler manifold and let $f\colon X\dashrightarrow X$ be a dominant meromorphic map as in Section \ref{Introduction}.
  If $f$ is a holomorphic map, then
  \begin{align*}
   (f^n)^{*}=(f^{*})^n:\mathrm{H}^{k,k}(X)\rightarrow\mathrm{H}^{k,k}(X)
  \end{align*}
  holds for every $0\leq k\leq g$ and so the dynamical degrees of $f$ is calculated as
  \begin{align*}
   \lambda_k(f)=\lim_{n\to+\infty}||(f^n)^{*}:\mathrm{H}^{k,k}(X)\rightarrow\mathrm{H}^{k,k}(X)||^{\frac{1}{n}}&=\lim_{n\to+\infty}||(f^{*})^n:\mathrm{H}^{k,k}(X)\rightarrow\mathrm{H}^{k,k}(X)||^{\frac{1}{n}}\\
                                                                                                               &=\rho(f^{*}:\mathrm{H}^{k,k}(X)\rightarrow\mathrm{H}^{k,k}(X))
  \end{align*}
  where $\rho(f^{*})$ is the spectral radius of $f^{*}:\mathrm{H}^{k,k}(X)\rightarrow\mathrm{H}^{k,k}(X)$.
  Under this condition, $\lambda_g(f)$ is equal to the number of points in $f^{-1}(a)$ for a general point $a$ in $X$.
  Thus, if $f$ is an automorphism of a $g$-dimensional complex projective variety $X$, then $\lambda_g(f)=1$ (cf.\ \cite{DS17}).\par
  Especially, let $X$ be a $g$-dimensional abelian variety over $\mathbb{C}$ and write as $V/\Lambda$ where $V=\mathbb{C}^g$ and $\Lambda$ a $\mathbb{Z}$-lattice in $V$.
  In this paper, all abelian varieties are considered over $\mathbb{C}$ and we abbreviate the base field from here.
  An abelian variety $X$ is called \textit{simple} if it contains no abelian subvariety except $0$ and $X$.\par
  Let $f\colon X \longrightarrow X$ be a morphism of an abelian variety $X$.
  Then, $f$ can be decomposed as $f=t\circ f'$ where $t$ is a translation map and $f'\colon X \longrightarrow X$ is a morphism of $X$ such that $f'(0)=0$ and $f'(x+y)=f'(x)+f'(y)$ for any $x,y\in X$.
  A morphism of $X$ which preserves the group structure of $X$ is called \textit{endomorphism} in this paper.
  Also, an endomorphism of $X$ is called \textit{automorphism} if its inverse exists and also it is an endomorphism.
  For a translation map $t$ of a complex torus $X$, $t_*=\mathrm{id}:\mathrm{H}_1(X)\longrightarrow\mathrm{H}_1(X)$ holds on singular homology and this implies $(t\circ f')^*=f'^*\circ t^*=f'^*$ on $\mathrm{H}^{k,k}(X)$ and so 
  \begin{align*}
   \lambda_k(t\circ f')=\lambda_k(f').
  \end{align*}
  Thus, the calculation of the dynamical degrees of morphisms of $X$ can be reduced to that of endomorphisms of $X$ and so from here, we consider only the dynamical degrees of endomorphisms of $X$.\par
  Denote the set of endomorphisms of $X$ by $\mathrm{End}(X)$ with the natural ring structure, and is called the endomorphism ring of $X$.
  Also the endomorphism algebra of $X$ is defined as $\mathrm{End}_\mathbb{Q}(X):=\mathrm{End}(X)\otimes_\mathbb{Z} \mathbb{Q}$.
  If $X$ is a simple abelian variety, then the endomorphism algebra $\mathrm{End}_\mathbb{Q}(X)$ is a division algebra of finite dimension over $\mathbb{Q}$.
  \begin{thm}[{Poincar\'e's Complete Reducibility Theorem (\cite[Theorem 5.3.7]{BL04})}]
   Let $X$ be a complex abelian variety.
   Then there is an isogeny
   \begin{align*}
    X\rightarrow X_1^{n_1}\times\cdots\times X_r^{n_r},
   \end{align*}
   where each $X_i$ is a simple abelian variety and $n_i\in\mathbb{Z}_{>0}$.\par
   Moreover, $X_i$ and $n_i$ are unique up to isogenies and permutations.
  \end{thm}
  \begin{remark}[{cf.\ \cite[Corollary 5.3.8]{BL04}}]\label{Poincare remark}
   By using the fact that the endomorphism algebra of a simple abelian variety is divisional, the above isogeny induces
   \begin{align*}
    \mathrm{End}_\mathbb{Q}(X)\simeq \mathrm{M}_{n_1}(F_1)\oplus\cdots\oplus\mathrm{M}_{n_r}(F_r),
   \end{align*}
   where $F_i=\mathrm{End}_\mathbb{Q}(X_i)$ is a division ring.\par
   Moreover, for an abelian variety $X$, $\mathrm{End}_\mathbb{Q}(X)$ is divisional if and only if $X$ is a simple abelian variety.
  \end{remark}
  For an abelian variety $X=V/\Lambda$, $f\in\mathrm{End}(X)$ induces a natural analytic representation $\rho_a(f):V\longrightarrow V$ and a natural rational representation $\rho_r(f):\Lambda\longrightarrow\Lambda$.
  $\rho_a(f)$ is identified as an element of $\mathrm{M}_g(\mathbb{C})$ and $\rho_r(f)$ is identified as an element of $\rm{M}_{2g}(\mathbb{Z})$.\par
  Let $\rho_1, \rho_2, \ldots, \rho_{g}$ be the eigenvalues of $\rho_a(f)$. Then $\rho_1, \overline{\rho_1}, \ldots, \rho_{g}, \overline{\rho_{g}}$ are the eigenvalues of $\rho_r(f)$.\par
  Assume $f$ is an automorphism, or more generally, a surjective endomorphism.
  Then, $\rho_a(f)$ is an isomorphism and then by renumbering $\rho_1, \rho_2, \ldots, \rho_{g}$ as $\abs{\rho_1}\geq\abs{\rho_2}\geq\cdots\geq\abs{\rho_g}>0$, the $k$-th dynamical degree $\lambda_k(f)$ of $f$ is equal to $\prod_{i=1}^{k}\abs{\rho_i}^2$ (see e.g. \cite[Section 1.3]{DH22}).\par
  The following theorems are useful for considerations.
  \begin{theorem}[{\cite[Chapter 13.1]{BL04}}]\label{Fix1}
   Let $X$ be an abelian variety and let $f$ be an endomorphism of $X$.
   Then
   \begin{align*}
    \#{\mathrm{Fix}(f)}=\mleft|\mathrm{det}({\id}_V-\rho_a(f))\mright|^2=\mathrm{det}({\id}_\Lambda-\rho_r(f))
                       =\prod_{i=1}^{g}(1-\rho_{i})(1-\overline{\rho_{i}}).
   \end{align*}
  \end{theorem}
  \begin{remark}[{cf.\ \cite[Chapter 13.1]{BL04}}]
   Here, $\#{\mathrm{Fix}(f)}$ is defined for a holomorphic map $f$ of a complex torus $X$ as 
  \begin{align*}
    \#{\mathrm{Fix}(f)}:=
     \begin{cases}
      \text{cardinarity of }\mathrm{Fix}(f) & \mathrm{dim}(\mathrm{Fix}(f))=0, \\
      0                                     & \mathrm{dim}(\mathrm{Fix}(f))>0,
     \end{cases}
   \end{align*}
   where $\mathrm{Fix}(f)$ is the analytic subvariety of $X$ consisting of the fixed points of $f$.
   $\#{\mathrm{Fix}(f)}$ is invariant under translation maps (cf.\ \cite[Lemma 13.1.1]{BL04}).\par
   For $f\in\mathrm{End}(X)$, if $\mathrm{Fix}(f)$ has a positve dimension, then some eigenvalue of $\rho_r(f)$ is equal to $1$ and so $\#{\mathrm{Fix}(f)}=0$.
  \end{remark}
  Assume $X$ is a simple abelian variety so that the endomorphism algebra $B:=\mathrm{End}_\mathbb{Q}(X)$ is a division ring of finite dimension over $\mathbb{Q}$ as above.
  The division ring $B$ has a field $F$ as a center with $\left[B\colon F\right]=d^2$ (cf.\ \cite[\S5]{Dra83}), $\left[F\colon \mathbb{Q}\right]=e$.
  Under this notation, the next theorem holds.
  \begin{theorem}[{\cite[Chapter 13.1]{BL04}}]\label{Fix2}
   Identify $f\in\mathrm{End}(X)$ as an element of $B:=\mathrm{End}_\mathbb{Q}(X)$ for a simple abelian variety $X$.
   Then,
   \begin{align*}
    \#{\mathrm{Fix}(f)}=\mathrm{Nrd}_{B/\mathbb{Q}}({\id}_X-f)^{\frac{2g}{de}},
   \end{align*}
   where $\mathrm{Nrd}_{B/\mathbb{Q}}:B\rightarrow \mathbb{Q}$ is the reduced norm map for central simple algebras.
  \end{theorem}
  \begin{remark}\label{finite-dimensional central simple algebra}
   For a finite-dimensional central simple algebra $B$ over a field $K$, the reduced norm $\mathrm{Nrd}_{B/K}$ 
   is calculated as below (cf.\ \cite[\S 22]{Dra83}).
   There exists a finite field extension $K'\supset K$ with $B\otimes_K K'\simeq\mathrm{M}_n(K')$ and fix an isomorphism $\phi\colon B\otimes_K K'\rightarrow\mathrm{M}_n(K')$. 
   Then $\mathrm{Nrd}_{B/K}:B\rightarrow K$ 
   is defined as
   \begin{align*}
    \mathrm{Nrd}_{B/K}(a):=\mathrm{det}(\phi(a\otimes1))\ (a\in B),
   \end{align*}
   and this is independent of the field extension of $K$ and the choice of the isomorphism.\par
   Also, for a finite field extension $K\supset k$, the reduced norm $\mathrm{Nrd}_{B/k}$ 
   is defined as 
   \begin{align*}
    \mathrm{Nrd}_{B/k}(a):=\mathrm{N}_{K/k}(\mathrm{Nrd}_{B/K}(a))\ (a\in B),
   \end{align*}
   where $\mathrm{N}_{K/k}$ is the field norm for the field extension $K\supset k$ by considering $K$ as a vector space over $k$.
   The details of reduced norms are explained in Section \ref{Algebraic preparation}.
  \end{remark}

 \subsection{Algebraic preparations}\label{Algebraic preparation}
  This subsection is devoted to the notations and the properties which are used later from Section \ref{Endomorphism of simple abelian variety}.
  \begin{flushleft}{\bf{Salem numbers, Pisot numbers}}\end{flushleft}
   \begin{definition}[{\cite[Chapter 5.2]{BDGGH+92}}]\label{Salem Pisot}
    A Salem number is a real algebriac integer $\lambda$ greater than $1$ where its other conjugates have modulus at most equal to $1$, at least one having a modulus equal to $1$.
    A Pisot number is a real algebraic integer $\lambda$ greater than $1$ where its all other conjugates have modulus less than $1$.
   \end{definition}\par
   For a Salem number $\lambda$, the minimal polynomial $P(x)\in\mathbb{Z}[x]$ of $\lambda$ is called the Salem polynomial of $\lambda$.
   By some deductions, $\frac{1}{\lambda}$ is the only root of $P(x)$ whose modulus is less than $1$ and the set of the roots of $P(x)$ can be written as $\mleft\{\lambda,\frac{1}{\lambda}\mright\}\cup\{z_1,\overline{z_1},\ldots z_k,\overline{z_k}\}$ where the $z_i$ are all of modulus $1$ (cf.\ \cite[p.84]{BDGGH+92}).\par
   Thus, the definition of Salem numbers can be rewritten as below.
   \begin{definition}\label{Definition of Salem}
    A Salem number is a real algebraic integer $\lambda>1$ of degree at least $4$ such that its minimal polynomial $P(x)$ has $\lambda$, $\frac{1}{\lambda}$ as its roots and all other roots have modulus $1$.
   \end{definition}
   \begin{remark}
    Let $\lambda$ be a Salem number and let $g$ be the degree of $\lambda$.
    From above, $g$ is always even, and so the degree of a Salem number is always even.\par
    Also by this defintion, powers of Salem numbers are again Salem numbers.
   \end{remark}

  \begin{flushleft}{\bf{CM-fields}}\end{flushleft}\par
   A number field $K$ is called \textit{totally real} if the image of every $\mathbb{Q}$-embedding $\sigma:K\hookrightarrow\mathbb{C}$ is inside $\mathbb{R}$.
   Also, a number field $K$ is called \textit{totally complex} if the image of every $\mathbb{Q}$-embedding $\sigma:K\hookrightarrow\mathbb{C}$ is not inside $\mathbb{R}$.
   \begin{definition}
    A CM-field is a number field $K$ which satisfies the following conditions.
    \begin{enumerate}
     \item $K$ is a totally complex field.
     \item $K$ is a quadratic extension of some totally real number field.
    \end{enumerate}
   \end{definition}

  \begin{flushleft}{\bf{Quaternion algebras}}\end{flushleft}\par
   For a field $F$ whose characteristic is not $2$, a quaternion algebra over $F$ is defined as an algebra $B$ over $F$ which has a basis $1,i,j,ij$ over $F$ with $i^2=a, j^2=b, ij=-ji$ ($a, b\in F^{\times}$).
   Under this condition, $B$ is written as $\left(\frac{a,b}{F}\right)$ and $1,i,j,ij$ is called an $F$-basis for $B$ (cf.\ \cite[Chapter 2]{Voi21}).\par
   There are some properties for quaternion algebras.
   \begin{itemize}
    \item $\left(\frac{a,b}{F}\right)\simeq\left(\frac{b,a}{F}\right)$
    \item $\left(\frac{a,b}{F}\right)\simeq\left(\frac{aa',b}{F}\right)$ ($a'\in F^{\times2}$)
    \item For a field extension $F\subset K$, $\left(\frac{a,b}{F}\right)\otimes_F K\simeq\left(\frac{a,b}{K}\right)$.
    \item When $F=\mathbb{R}$, $\left(\frac{1,1}{\mathbb{R}}\right)\simeq\left(\frac{1,-1}{\mathbb{R}}\right)\simeq\mathrm{M}_2(\mathbb{R})$ and $\left(\frac{-1,-1}{\mathbb{R}}\right)\simeq\mathbb{H}$.
   \end{itemize}\par
   Let $B$ be a quaternion algebra over a totally real number field $F$.
   $B$ is called \textit{totally indefinite} if $B\otimes_\sigma \mathbb{R}\simeq\mathrm{M}_2(\mathbb{R})$ for every embedding $\sigma\colon F\hookrightarrow\mathbb{R}$.
   On the other hand, $B$ is called \textit{totally definite} if $B\otimes_\sigma\mathbb{R}\simeq\mathbb{H}$ for every embedding $\sigma\colon F\hookrightarrow\mathbb{R}$.
   By the above properties, $B=\left(\frac{a,b}{F}\right)$ is totally indefinite if and only if either $\sigma(a)>0$ or $\sigma(b)>0$ holds for any $\mathbb{Q}$-embedding $\sigma\colon F\hookrightarrow\mathbb{C}$.\par
   Especially, if $a\in\mathbb{Q}_{>0}$, then $B=\left(\frac{a,b}{F}\right)$ is totally indefinite.

  \begin{flushleft}{\bf{Anti-involutions}}\end{flushleft}
   \begin{definition}[{\cite[Definition 3.2.1]{Voi21}}]
    For a field $F$ and an algebra $B$ over $F$, an $F$-linear map $\phi:B\rightarrow B$ which satisfies the following conditions is called an \textit{anti-involution} over $F$. 
    \begin{enumerate}
     \item $\phi(1)=1$
     \item $\phi(\phi(x))=x$ ($x\in B$)
     \item $\phi(xy)=\phi(y)\phi(x)$ ($x, y\in B$)
   \end{enumerate}
   \end{definition}
   \begin{definition}[{cf.\ \cite[Definition 8.4.1]{Voi21}}]
    For a subfield $F\subset\mathbb{R}$ and a finite-dimensional algebra $B$ over $F$, an anti-involution $\phi:B\rightarrow B$ over $F$ is called \textit{positive} over $F$ if $\mathrm{Tr}(\phi(x)x)>0$ for any nonzero $x\in B$.
    Here, for an element $a\in B$, $\mathrm{Tr}(a)$ is defined as the trace of the left multiplication map $a:B\rightarrow B$ where $B$ is considered as a vector space over $F$.
   \end{definition}
   \begin{remark}
    Let $B$ be an algebra over a number field $F$ and $\phi:B\rightarrow B$ be an anti-involution over $F$.
    Then, $\phi:B\rightarrow B$ is also an anti-involution over $\mathbb{Q}$.
    But, a positive anti-involution over $F$ is not always a positive anti-involution over $\mathbb{Q}$.
    In this paper, positivity of anti-involutions is always considered over $\mathbb{Q}$.
   \end{remark}
   \begin{example}\label{anti-involution}
    For a quaternion algebra $B=\mleft(\frac{a,b}{F}\mright)$ with an $F$-basis $1,i,j,ij$, the quaternion conjugate on $B$ is the map 
    \begin{align*}
     x_1+x_2 i+x_3 j+x_4 ij\longmapsto x_1-x_2 i-x_3 j-x_4 ij\ (x_i \in F)
    \end{align*}
    which is an anti-involution over $F$ and denote this map by $x\mapsto \overline{x}$ (\cite[Chapter 3.2]{Voi21}).
   \end{example}\par
   Let $B$ be a central simple algebra over a field $K$ and $\sigma$ be an anti-involution on $B$ over some field, which may be other than $K$.
   For $a\in K$, $\sigma(a)$ is an element of $K$ since for all $b\in B$,
   \begin{align*}
    \sigma(a)b=\sigma(\sigma^{-1}(b)a)=\sigma(a\sigma^{-1}(b))=b\sigma(a),
   \end{align*}
   and so $\sigma$ can be restricted to $K$.
   \begin{definition}\label{second kind}
    Let $B$ be a central simple algebra over a field $K$.\par
    An anti-involution $\sigma$ on $B$ is called \textit{first kind} if $\sigma$ fixes $K$ pointewise.
    Otherwise, $\sigma$ is called \textit{second kind}.
   \end{definition}
   \begin{theorem}[{cf.\ \cite[Theorem 5.5.3]{BL04}}]\label{construction of positive anti-involution}
    Let $B$ be a totally indefinite quaternion algebra of finite dimension over $\mathbb{Q}$ with center a totally real number field $K$.
    Assume $B$ is divisional.
    Then, a positive anti-involution $\phi:B\rightarrow B$ over $\mathbb{Q}$ can be written as
    \begin{align*}
     \phi(x)=c^{-1}\overline{x}c  
    \end{align*}
    where $c\in B\backslash K$ with $c^2\in K$ totally negative (i.e., the conjugates are all real and negative). 
   \end{theorem}
   \begin{theorem}[{cf.\ \cite[Chapter 21]{Mum70}, \cite[Theorem 5.5.6]{BL04}}]\label{existence of positive anti-involution}
    Let $B$ be a division algebra of finite dimension over $\mathbb{Q}$ with center a CM-field $K$.\par
    Assume that $B$ admits an anti-involution $x\mapsto\tilde{x}$ of the second kind.
    Then there exists a positive anti-involution $x\mapsto x'$ of the second kind.
   \end{theorem}

  \begin{flushleft}{\bf{Orders}}\end{flushleft}
   \begin{definition}
    Let $R$ be an integral domain and define $K=\mathrm{Frac}(R)$.
    For a finite-dimensional algebra $B$ over K, a subset $\mathcal{O}\subset B$ which satisfies the following conditions is called an $R$-order.
     \begin{enumerate}
     \item $\mathcal{O}$ is a finitely generated $R$-submodule of $B$
     \item $\mathcal{O}$ spans $B$ over K
     \item $\mathcal{O}$ is closed under multiplication induced from $B$
    \end{enumerate}
   \end{definition}\par
   Often, the definition of an $R$-order is applied for the case $R=\mathbb{Z}$, $K=\mathbb{Q}$.
   \begin{examples}
   Here are some examples of a $\mathbb{Z}$-order (cf.\ \cite[Chapter 10]{Voi21}, \cite[Chapter 1]{BL04}).
    \begin{itemize}
     \item For a number field $K$, the ring of integer $\mathcal{O}_K$ is a $\mathbb{Z}$-order of $K$.
           Also, $\mathcal{O}_K$ is the maximal $\mathbb{Z}$-order with respect to the inclusion.
     \item For a quaternion algebra $\left(\frac{a,b}{\mathbb{Q}}\right)$ ($a,b\in\mathbb{Z}\backslash\{0\}$) with a $\mathbb{Q}$-basis $1,i,j,ij$, $\mathbb{Z}\oplus\mathbb{Z}i\oplus\mathbb{Z}j\oplus\mathbb{Z}ij$ is a $\mathbb{Z}$-order of $\left(\frac{a,b}{\mathbb{Q}}\right)$. 
     \item For an abelian variety $X$, $\mathrm{End}(X)$ is a $\mathbb{Z}$-order of $\mathrm{End}_\mathbb{Q}(X)$.
    \end{itemize}
   \end{examples}

  \begin{flushleft}{\bf{Reduced norms, Reduced characteristic polynomials}}\end{flushleft}\par
   Let $B$ be a finite-dimensional central simple algebra over a number field $K$ with $[B:K]=d^2$ for some integer $d\in\mathbb{Z}_{>0}$.\par
   As in Remark \ref{finite-dimensional central simple algebra}, for some field $K'$ with the finite Galois extension $K'/K$, an isomorphism
   \begin{align*}
    \phi\colon B\otimes_K K'\stackrel{\simeq}{\longrightarrow}\mathrm{M}_d(K')
   \end{align*}
   is induced.
   For $\alpha\in B$, $\mathrm{Nrd}_{B/K}(\alpha)$ is calculated as $\mathrm{det}(\phi(\alpha\otimes1))$.
   Define
   \begin{align*}
    p_\alpha(n)=\mathrm{Nrd}_{B/K}(n-\alpha)=\mathrm{det}(nI-\phi(\alpha\otimes1))
   \end{align*}
   as a polynomial in $n$.
   $p_\alpha(n)$ is invariant under all $\sigma\in \mathrm{Gal}(K'/K)$ and so the coefficients of $p_\alpha(n)$ are all in $K$.
   Thus, $p_\alpha(x)\in K[x]$ and this polynomial is called the reduced characteristic polynomial of $\alpha$.
   Let $P(x)\in K[x]$ be the minimal polynomial of $\alpha$ over $K$ and also this is the minimal polynomial of $\alpha\otimes1\in B\otimes_K K'$ and so of $\phi(\alpha\otimes1)\in\mathrm{M}_d(K')$.
   Thus, the roots of the reduced characteristic polynomial $p_\alpha(x)\in K[x]$ are the same as the roots of the minimal polynomial $P(x)$ by the theorem of linear algebra.
   Moreover, because of the minimality, the reduced characteristic polynomial can be written as $p_\alpha(x)=P(x)^s$ for some $s\in\mathbb{Z}_{>0}$.

  \begin{flushleft}{\bf{Dedekind domains}}\end{flushleft}\par
   Let $R$ be an integral domain and $K$ be its fraction field.\par
   A fractional ideal of $R$ is a non-zero $R$-submodule $I$ of $K$ such that there exists $a\in R$ such that $aI\subset R$.
   A fractional ideal $I$ of $R$ is called \textit{invertible} if there exists a fractional ideal $J$ of $R$ such that $IJ=R$.
   \begin{definition}[{cf.\ \cite[Vol 2, Chapter 9.5]{Coh89}}]
    A Dedekind domain is an integral domain $R$ which satisfies the following equivalent conditions.
    \begin{enumerate}
    \renewcommand{\labelenumi}{(\roman{enumi})}
     \item Every non-zero ideal $I\subsetneq R$ is invertible.
     \item Every non-zero ideal $I\subsetneq R$ is expressed as a finite product of the prime ideals uniquely.
     \item $R$ is Noetherian, integrally closed and all non-zero prime ideals are maximal ideals.
    \end{enumerate}
   \end{definition}
   Let $\mathcal{O}$ be a Dedekind domain and $F$ be its fraction field.
   Let $K/F$ be a finite separable field extension and $\mathcal{O}'$ be the integral closure of $\mathcal{O}$ in $K$.
   Then $\mathcal{O}'$ is a Dedekind domain with the fraction field $K$ (cf.\ \cite[Chapter I, Proposition 8.1]{Neu99}).\par
   For the pair $(\mathcal{O},F,\mathcal{O}',K)$ constructed just now and a non-zero prime ideal $\mathfrak{p}\subset\mathcal{O}$, there is the factorization $\mathfrak{p}\mathcal{O}'=\mathfrak{P}_1^{e_1}\cdots\mathfrak{P}_g^{e_g}$ since $\mathcal{O}'$ is a Dedekind domain.\par
   The next conditons are equivalent and when they hold, we say that $\mathfrak{P}$ is over $\mathfrak{p}$.
   \begin{itemize}
    \item $\mathfrak{P}\cap\mathcal{O}=\mathfrak{p}$
    \item $\mathfrak{P}$ appears in the prime ideal factorization of $\mathfrak{p}\mathcal{O}'$
   \end{itemize}\par
   Take the pair $(\mathcal{O},F,\mathcal{O}',K)$ again and assume $K/F$ is a separable extension of degree $n$.
   For a non-zero prime ideal $\mathfrak{p}\subset\mathcal{O}$, there is the prime ideal factorization $\mathfrak{p}\mathcal{O}'=\mathfrak{P}_1^{e_1}\cdots\mathfrak{P}_g^{e_g}$ with $f_i=[\mathcal{O}'/\mathfrak{P}_i:\mathcal{O}/\mathfrak{p}]$ and then
   \begin{align*}
    n=\sum_{i=1}^g e_if_i.
   \end{align*}
   Under this condition, some notations can be defined (cf.\ \cite[Chapter I, \S8]{Neu99}).
   \begin{definition}
    \begin{enumerate}
     \item $\mathfrak{p}$ is \textit{totally split} in $K$ if $e_i=1$, $f_i=1$ and $g=n$.
     \item $\mathfrak{p}$ is \textit{nonsplit} in $K$ if $g=1$.
     \item $\mathfrak{P}_i$ is \textit{unramified} over $F$ if $e_i=1$ and the field extension $\mathcal{O}'/\mathfrak{P}_i\supset\mathcal{O}/\mathfrak{p}$ is separable.
     \item $\mathfrak{p}$ is \textit{unramified} in $K$ if all $\mathfrak{P}_i$ are unramified over $F$.
    \end{enumerate}
   \end{definition}\par
   Moreover, by assuming $K/F$ is a Galois extension, 
   \begin{align*}
   e_1=\cdots=e_g=e, f_1=\cdots=f_g=f
   \end{align*}
    hold and this implies $n=efg$.
    This theorem is called Hilbert's ramification theory (cf.\ \cite[Chapter I, \S9]{Neu99}).\par
    On the same assumption,
   \begin{align*}
    \mathcal{D}_\mathfrak{P}:=\{\sigma\in \mathrm{Gal}(K/F)\mid\sigma(\mathfrak{P})=\mathfrak{P}\}
   \end{align*}
   is called the decomposition group of $\mathfrak{P}$ and its order is $ef$.\par
   Moreover, if $\mathfrak{P}$ is over a prime ideal $\mathfrak{p}\subset\mathcal{O}$, then $(\mathcal{O}'/\mathfrak{P})/(\mathcal{O}/\mathfrak{p})$ is a normal extension and there is a natural surjective group homomorphism 
   \begin{align*}
    \mathcal{D}_\mathfrak{P}\rightarrow\mathrm{Aut}\left((\mathcal{O}'/\mathfrak{P})/(\mathcal{O}/\mathfrak{p})\right).
   \end{align*}
   Assume $(\mathcal{O}'/\mathfrak{P})/(\mathcal{O}/\mathfrak{p})$ is separable, then this homomorphism is isomorphic if and only if $\mathfrak{P}$ is unramified over $F$.

  \begin{flushleft}{\bf{Cyclotomic polynomials}}\end{flushleft}\par
   In this paper, for a positive integer $n$, let $\zeta_n$ be the primitive $n$-th root which has a smallest positive angle.\par
   Let $\Phi_n(x)\in\mathbb{Z}[x]$ be the minimal polynomial of $\zeta_n$ over $\mathbb{Q}$.
   This is called a cyclotomic polynomial.
   The degree of $\Phi_n(x)$ is Euler's totient function $\varphi(n)$.
   Also let $\Psi_n(x)$ be the minimal polynomial of $\zeta_n+\frac{1}{\zeta_n}=2\mathrm{cos}\left(\frac{2\pi}{n}\right)$ over $\mathbb{Q}$ (define $\Psi_4(x)=x$ for the case $n=4$).\par
   By the definition, for $n\geq3$, the equation
   \begin{align*}
    \Phi_n(x)=x^{\frac{\varphi(n)}{2}}\Psi_n\left(x+\frac{1}{x}\right)
   \end{align*}
   holds and so $\Psi_n(x)$ has degree $\frac{\varphi(n)}{2}$.\par
   Also, there is a theorem for the constant term of $\Psi_n(x)$.
   \begin{theorem}[{cf.\ \cite{ACR16}}]\label{constant term of minimal polynomial}
    The absolute value of the constant term of $\Psi_n(x)$ is equal to $1$ except the following cases.
    \begin{enumerate}
    \renewcommand{\labelenumi}{\rm{(\roman{enumi})}}
     \item $\abs{\Psi_n(0)}=0$ if $n=4$
     \item $\abs{\Psi_n(0)}=2$ if $n=2^m$ with $m\in\mathbb{Z}_{\geq0}\backslash\{2\}$
     \item $\abs{\Psi_n(0)}=p$ if $n=4p^k$ with $k\in\mathbb{Z}_{>0},\ p\text{ an odd prime number}$
    \end{enumerate}
   \end{theorem}

  \begin{flushleft}{\bf{Dirichlet density, \v{C}ebotarev's density theorem}}\end{flushleft}\par
   For a number field $K$ and its ring of integer $\mathcal{O}_K$, a prime ideal of $\mathcal{O}_K$ is often said as a prime ideal of a field $K$.
   \begin{definition}[{\cite[Chapter I\hspace{-.1em}I\hspace{-.1em}I \S 1, Chapter V\hspace{-.1em}I\hspace{-.1em}I Definition 13.1]{Neu99}}]
    Let $M$ be a set of non-zero prime ideals of a number field $K$.
    Denote
    \begin{align*}
     \mathfrak{N}(\mathfrak{p})=p^{f}
    \end{align*}
    for a prime ideal $\mathfrak{p}\subset\mathcal{O}_K$ where $p\mathbb{Z}=\mathfrak{p}\cap\mathbb{Z}$ and $f=[\mathcal{O}_K/\mathfrak{p}:\mathbb{Z}/p\mathbb{Z}]$.\par
    Under this notation, the limit
    \begin{align*}
     d(M)=\lim_{s\to1+0}\frac{\sum_{\mathfrak{p}\in M} \mathfrak{N}(\mathfrak{p})^{-s}}{\sum_{\mathfrak{p}} \mathfrak{N}(\mathfrak{p})^{-s}},
    \end{align*}
    where the denominater is the sum for all non-zero prime ideals, is called the Dirichlet density of $M$ if it exists.
   \end{definition}
   \begin{remark}\label{infinite property}
    In the above definition, the denominator $\sum_{\mathfrak{p}} \mathfrak{N}(\mathfrak{p})^{-s}$ diverges to $+\infty$ as \cite[Chapter V\hspace{-.1em}I\hspace{-.1em}I \S 13]{Neu99}.\par
    Thus, if  $M$ is a finite set of prime ideals of $K$, then $d(M)=0$.
   \end{remark}
   \begin{definition}[{cf.\ \cite[Chapter 6.3]{Sam70}}]\label{Frobenius automorphism}
    Let $K/F$ be a Galois extension of number fields with Galois group $G=\mathrm{Gal}(K/F)$ and let $\mathfrak{P}$ be a prime ideal of $K$ which is unramified over $F$.
    Denote $\mathfrak{p}=\mathfrak{P}\cap\mathcal{O}_F$.\par
    Then there is one and only one $\sigma\in G$ such that $\sigma$ induces the map $a\mapsto a^s$ on $\mathcal{O}_K/\mathfrak{P}$ where $s=\#(\mathcal{O}_F/\mathfrak{p})$.
    This $\sigma$ is denoted by $\left(\frac{K/F}{\mathfrak{P}}\right)$ and called the Frobenius automorphism.
   \end{definition}
   \begin{remark}
    On the notation in Definition \ref{Frobenius automorphism}, $\sigma(\mathfrak{P})=\mathfrak{P}$ and by considering the isomorphism $\mathcal{D}_\mathfrak{P}\rightarrow\mathrm{Aut}\left((\mathcal{O}_K/\mathfrak{P})/(\mathcal{O}_F/\mathfrak{p})\right)$, $\mathcal{D}_\mathfrak{P}$ is generated by $\sigma$.
   \end{remark}
   \begin{definition}[{\cite[Chapter V\hspace{-.1em}I\hspace{-.1em}I \S 13]{Neu99}}]
    Let $K/F$ be a Galois extension of number fields with Galois group $G$ and take $\sigma\in G$.\par
    Define $P_{K/F}(\sigma)$ as the set of prime ideals $\mathfrak{p}$ of $F$ unramified in $K$ such that there is a prime ideal $\mathfrak{P}$ of $K$ over $\mathfrak{p}$ such that the Frobenius automorphism $\left(\frac{K/F}{\mathfrak{P}}\right)$ is equal to $\sigma$.
   \end{definition}
   The density of $P_{K/F}(\sigma)$ is calculated by \v{C}ebotarev's density theorem.
   \begin{theorem}[{cf.\ \cite[Chapter V\hspace{-.1em}I\hspace{-.1em}I Theorem 13.4]{Neu99}}]\label{Cebotarev's density theorem}
    Let $K/F$ be a Galois extension of number fields with Galois group $G$.
    Then for every $\sigma\in G$, the set $P_{K/F}(\sigma)$ has a density, and it is given by
    \begin{align*}
     d(P_{K/F}(\sigma))=\frac{\#\langle\sigma\rangle}{\# G},
    \end{align*}
    where $\langle\sigma\rangle:=\{\tau^{-1}\sigma\tau\mid\tau\in G\}$.\par
    In particular, by Remark \ref{infinite property}, $P_{K/F}(\sigma)$ consits of infinitely many prime ideals.
   \end{theorem}

\section{Endomorphisms of simple abelian varieties}\label{Endomorphism of simple abelian variety}
 This section is devoted to examining endomorphisms of simple abelian varieties and their dynamical degrees in detail.
 \subsection{Endomorphism algebras of simple abelian varieties}\label{Endomorphism algebra of simple abelian varieties}
  Let $X$ be a $g$-dimensional simple abelian variety and define $B:=\mathrm{End}_\mathbb{Q}(X):=\mathrm{End}(X)\otimes_\mathbb{Z}\mathbb{Q}$ as in Section \ref{Dynamical degree}.
  $B$ is a division ring of finite dimension over $\mathbb{Q}$ and the Rosati involution $'$ on $B=\mathrm{End}_\mathbb{Q}(X)$ is a positive anti-involution.
  Also, $B$ has a field $K$ as a center and the anti-involution $'$ can be restricted to $K$.
  By defining $K_0:=\{x\in K\mid x'=x\}$, $K_0$ be a totally real number field and either $K=K_0$ or $K$ is a totally complex quadratic extension of $K_0$.
  Denote $[B:K]=d^2$, $[K:\mathbb{Q}]=e$ and $[K_0:\mathbb{Q}]=e_0$.
  In summary, the endomorphism algebra $B=\mathrm{End}_\mathbb{Q}(X)$ can be classified as below (cf.\ \cite[Chapter 5.5]{BL04}).
  \begin{table}[hbtp]
   \caption{Classification of $\mathrm{End}_\mathbb{Q}(X)$}
   \label{table}
   \begin{tabular}{c|c|c|c|c|c}
    & $B=\mathrm{End}_\mathbb{Q}(X)$ & $K$ & $d$ & $e_0$ & restriction \\
    \hline\hline
    Type 1 & $K$                                             & totally real & $1$ & $e$           & $e\mid g$\\
    Type 2 & totally indefinite quaternion algebra over $K$  & totally real & $2$ & $e$           & $2e\mid g$\\
    Type 3 & totally definite quaternion algebra over $K$    & totally real & $2$ & $e$           & $2e\mid g$\\
    Type 4 & division ring with center $K$                   & CM-field     & $d$ & $\frac{e}{2}$ & $\frac{d^2 e}{2}\mid g$
   \end{tabular}
  \end{table}

 \subsection{Automorphisms}\label{Automorphism of simple abelian varieties}
  In this subsection, by using the above classification, we analyze automorphisms of a simple abelian variety $X$, in each type.\par
  Any endomorphism of $X$ can be considered as an element of $B=\mathrm{End}_{\mathbb{Q}}(X)$ and moreover it is an integral element of $B$.
  Thus, any automorphism of $X$ can be considered as an element $\alpha\in B$ such that $\alpha$ and $\alpha^{-1}$ are both integral elements of $B$.\par
  $B=\mathrm{End}_\mathbb{Q}(X)$, $K$, $K_0$, $d$, $e$ and $e_0$ are as in Section \ref{Endomorphism algebra of simple abelian varieties}.
  Define
  \begin{align*}
   U_K:=\{x\in\mathcal{O}_K\backslash\{0\}\mid x^{-1}\in\mathcal{O}_K\}
  \end{align*}
  as the group of the invertible elements in $\mathcal{O}_K$.
  \begin{flushleft}{\bf{Type 1}}\end{flushleft}\par
   Any automorphism of $X$ is regarded as an element $\alpha\in\mathcal{O}_K$ and its inverse is also in $\mathcal{O}_K$ and so $\alpha$ is in $U_K$.\par
   Let $F(x)=x^n+a_1x^{n-1}+\cdots+a_n\in\mathbb{Z}[x]$ be the minimal polynomial of $\alpha$, and then $a_nx^n+a_{n-1}x^{n-1}+\cdots+1$ has the root $\alpha^{-1}$.
   This polynomial would be the minimal polynomial of $\alpha^{-1}$ with some constant multiplication, and so $a_n=\pm1$.\par
   Thus, $\alpha$ is a root of a monic polynomial in $\mathbb{Z}[x]$ whose constant term is $\pm1$ and its roots are all real.
  \begin{flushleft}{\bf{Type 2, Type 3, Type 4}}\end{flushleft}\par
   Any automorphism of $X$ can be regarded as an element $\alpha\in B$ such that both $\alpha$ and $\alpha^{-1}$ are integral elements in $B$.
   This implies that the coefficients of the minimal polynomial of $\alpha$ (resp. $\alpha^{-1}$) over $K$ are all in $\mathcal{O}_K$ and so the constant term is in $U_K$ and this minimal polynomial is of degree at most $d$.\par
   Also, the minimal polynomial of $\alpha$ over $\mathbb{Q}$ is of the form $x^n+a_1x^{n-1}+\cdots\pm1\in\mathbb{Z}[x]$ as in Type 1.

 \subsection{Caluculations of dynamical degrees}\label{Calculation}
  This subsection is devoted to calculating the values of the first dynamical degree of surjective endomorphisms of simple abelian varieties.\par
  Let $X$ be a $g$-dimensional simple abelian variety.
  $K$, $K_0$, $d$, $e$ and $e_0$ are as in Section \ref{Endomorphism algebra of simple abelian varieties}.
  \begin{flushleft}{\bf{Type 1}}\end{flushleft}\par
   For a surjective endomorphism $\alpha\in\mathrm{End}(X)\subset\mathrm{End}_\mathbb{Q}(X)=K$, the minimal polynomial $F(x)\in\mathbb{Z}[x]$ of $\alpha$ of degree $d'$ has only real roots.
   Let $\rho_1,\rho_2,\ldots,\rho_g$ be the eigenvalues of $\rho_a(\alpha)$ and define the endomorphism $\phi:=1-(n-\alpha)\colon X\rightarrow X$ for some integer $n$.\par
   By applying Theorem \ref{Fix1} and Theorem \ref{Fix2} for $\phi:X\rightarrow X$,
   \begin{align*}
    \#{\mathrm{Fix}(1-(n-\alpha))}=\prod_{i=1}^{g}(n-\rho_{i})(n-\overline{\rho_{i}})\\
    \#{\mathrm{Fix}(1-(n-\alpha))}=\mathrm{Nrd}_{K/\mathbb{Q}}(n-\alpha)^{\frac{2g}{1\cdot e}}=\mathrm{N}_{K/\mathbb{Q}}(n-\alpha)^{\frac{2g}{e}}=&\mathrm{N}_{\mathbb{Q}(\alpha)/\mathbb{Q}}(n-\alpha)^{\frac{e}{d'}\cdot\frac{2g}{e}}
                                                                                                                                                 =&{F(n)}^{\frac{2g}{d'}}.
   \end{align*}
   Since these two formulas are equal for any integer $n\in\mathbb{Z}$,
   \begin{equation*}
    \prod_{i=1}^{g}(x-\rho_{i})(x-\overline{\rho_{i}})={F(x)}^{\frac{2g}{d'}}
   \end{equation*}
   as polynomials in $x$.
   Thus, each conjugate of $\alpha$ appears $\frac{2g}{d'}$ times in $\rho_1,\ldots,\rho_g,\overline{\rho_1},\ldots,\overline{\rho_g}$ and the dynamical degrees are calculated by using this.\par
   Especially, the first dynamical degree of $\alpha\colon X\rightarrow X$ is the square of the maximal absolute value of the roots of $F(x)$.
  \begin{flushleft}{\bf{Type 2, Type 3, Type 4}}\end{flushleft}\par
   A surjective endomorphism $\alpha\in\mathrm{End}(X)$ can be identified as an element which has the minimal polynomial of the form $G(x)\in\mathcal{O}_K[x]$ over $K$.
   Also, let $F(x)\in\mathbb{Z}[x]$ be the minimal polynomial of $\alpha$ over $\mathbb{Q}$.
   Denote the degree of $G(x)$ by $d'\leq d$ and the degree of $F(x)$ by $d''$.\par
   Denote the reduced characteristic polynomial of $\alpha\in B$ by $G'(x)$, whose degree is $d$ and as mentioned in Section \ref{Algebraic preparation}, $G'(x)$ can be written as $G'(x)=G(x)^\frac{d}{d'}$.
   Let $\rho_1,\rho_2,\ldots,\rho_g$ be the eigenvalues of $\rho_a(\alpha)$ and define the endomorphism $\phi:=1-(n-\alpha)\colon X\rightarrow X$ for some integer $n$.\par
   By applying Theorem \ref{Fix1} and Theorem \ref{Fix2} for $\phi:X\rightarrow X$,
   \begin{align*}
    \#{\mathrm{Fix}(1-(n-\alpha))}=\prod_{i=1}^{g}(n-\rho_{i})(n-\overline{\rho_{i}})
   \end{align*}
   \begin{align*}
    \#{\mathrm{Fix}(1-(n-\alpha))}=\mathrm{Nrd}_{B/\mathbb{Q}}(n-\alpha)^{\frac{2g}{de}}=\mathrm{N}_{K/\mathbb{Q}}(\mathrm{Nrd}_{B/K}(n-\alpha))^{\frac{2g}{de}}&=\mathrm{N}_{K/\mathbb{Q}}(G'(n))^{\frac{2g}{de}}\\
                                                                                                                                                                &=\prod_{i=1}^{e}(\sigma_i(G'(n)))^{\frac{2g}{de}}\\
                                                                                                                                                                &=\prod_{i=1}^{e}(\sigma_i(G(n)))^{\frac{2g}{d'e}}
   \end{align*}
   where $\left\{\sigma_i\right\}_{1\leq i\leq e}$ is the set of all $\mathbb{Q}$-embeddings $K\hookrightarrow\mathbb{C}$.
   Since these two formulas are equal for any integer $n\in\mathbb{Z}$,
   \begin{align*}
    \prod_{i=1}^{g}(x-\rho_{i})(x-\overline{\rho_{i}})&=\prod_{i=1}^{e}(\sigma_i(G(x)))^{\frac{2g}{d'e}}={F(x)}^{\frac{2g}{d''}}
   \end{align*}
   as polynomials in $x$.
   Thus, each conjugate of $\alpha$ appeares $\frac{2g}{d''}$ times in $\rho_1,\ldots,\rho_g,\overline{\rho_1},\ldots,\overline{\rho_g}$ and the dynamical degrees are calculated by using this.\par
   Especially, the first dynamical degree is the square of the maximal absolute value of roots of all over $\sigma_i(G(x))$.\\
   \par
   As a conclusion of Sections \ref{Automorphism of simple abelian varieties} and \ref{Calculation}, an automorphism $f$ of a simple abelian variety $X$ is identified as an element of the central simple division algebra $B=\mathrm{End}_\mathbb{Q}(X)$ over $K$.
   Denote this element by $\alpha$ and by considering the tower of extensions $K(\alpha)\supset K\supset \mathbb{Q}$, $\alpha$ has a minimal polynomial $F(x)\in\mathbb{Z}[x]$ over $\mathbb{Q}$ with the constant term $\pm1$.\par
   Also, the first dynamical degree of $f$ is calculated as the square of the maximal absolute value of the roots of $F(x)$.

\section{Constructions of simple abelian varieties}\label{Construction of simple abelian varieties}
 In Section \ref{Main theorem}, the next lemma proposed in \cite{DH22} is the key. This lemma is an immediate consequence of \cite[Chapter 9.4, Chapter 9.9]{BL04} composed with Remark \ref{Poincare remark}.
 \begin{lemma}[{\cite[Proposition 2.3]{DH22}}]\label{Construction}
  Let $B$ be a totally indefinite quaternion algebra over a totally real number field $F$ with $\left[F:\mathbb{Q}\right]=e$ and $'$ a positive anti-involution on $B$.
  Fix an order $\mathcal{O}$ in $B$ and suppose that $B$ is divisional.
  Then there exists a $2e$-dimensional simple abelian variety $X$ whose endomorphism ring $\mathrm{End}(X)$ contains $\mathcal{O}$. 
 \end{lemma}
 As this therorem, the next theorems can be deduced from \cite[Chapter 9]{BL04}.
 \begin{lemma}[{cf.\ \cite[Chapter 9.2]{BL04}}]\label{Construction1}
  Let $F$ be a totally real number field with $\left[F:\mathbb{Q}\right]=e$ and $'$ be a positive anti-involution on $F$ (e.g., $'=\mathrm{id}_F$).
  Fix an order $\mathcal{O}$ in $F$ and let $m$ be a positive integer.
  Then there exists an $em$-dimensional simple abelian variety $X$ with an isomorphism $F\stackrel{\simeq}{\longrightarrow}\mathrm{End}_{\mathbb{Q}}(X)$ which induces an injective ring homomorphism $\mathcal{O}\hookrightarrow\mathrm{End}(X)$.
 \end{lemma}
 \begin{remark}
  Consider the case $\mathcal{O}=\mathcal{O}_F$.
  Since the ring of integer $\mathcal{O}_F$ of $F$ is a maximal $\mathbb{Z}$-order of $F$ and $\mathrm{End}(X)$ is an order of $\mathrm{End}_\mathbb{Q}(X)$, the $em$-dimensional simple abelian variety $X$ in Lemma \ref{Construction1} satisfies $\mathrm{End}(X)\simeq\mathcal{O}_F$.
 \end{remark}
 \begin{lemma}[{cf.\ \cite[Chapter 9.4]{BL04}}]\label{Construction2}
  Let $B$ be a totally indefinite quaternion algebra over a totally real number field $F$ with $\left[F:\mathbb{Q}\right]=e$ and $'$ be a positive anti-involution on $B$.
  Fix an order $\mathcal{O}$ in $B$ and suppose that $B$ is divisional and let $m$ be a positive integer.
  Then there exists a $2em$-dimensional simple abelian variety $X$ with an isomorphism $B\stackrel{\simeq}{\longrightarrow}\mathrm{End}_{\mathbb{Q}}(X)$ which induces an injective ring homomorphism $\mathcal{O}\hookrightarrow\mathrm{End}(X)$.
 \end{lemma}
 \begin{lemma}[{cf.\ \cite[Proposition 2.4]{DH22}, \cite[Chapter 9.6]{BL04}}]\label{Construction4}
  Let $K$ be a CM-field with $\left[K:\mathbb{Q}\right]=e=2e_0$ for an integer $e_0\in\mathbb{Z}_{>0}$.
  Let $B$ be a central simple division algebra over $K$ with $[B:K]=d^2$ and $'$ be a positive anti-involution on $B$.\par
  Fix an order $\mathcal{O}$ in $B$ and let $m$ be a positive integer and assume one of the next conditions.
  \begin{enumerate}
   \item $dm\geq3$
   \item $dm=2$ and $e_0\geq2$
  \end{enumerate}
  Then there exists a $d^2e_0m$-dimensional simple abelian variety $X$ with an isomorphism $B\stackrel{\simeq}{\longrightarrow}\mathrm{End}_{\mathbb{Q}}(X)$ which induces an injective ring homomorphism $\mathcal{O}\hookrightarrow\mathrm{End}(X)$.
 \end{lemma}
 \begin{remark}
  The conditions for $d,e_0,m\in\mathbb{Z}_{>0}$ are come from the existence of $r_v,s_v\in\mathbb{Z}_{\geq0}$ ($1\leq v\leq e_0$) which satisfy the conditions in \cite[Chapter 9.6, Chapter 9.9]{BL04} as below.
  \begin{itemize}
   \item $r_v+s_v=dm$
   \item $\sum_{v=1}^{e_0}r_vs_v\neq0$
   \item $(r_v,s_v)\neq(1,1)$ for some $v$
  \end{itemize}
 \end{remark}

\section{Main Theorem}\label{Main theorem}
 This section is devoted to proving the next theorem which is analogous to Theorem \ref{DH22 theorem}. The proof is along to \cite{DH22}.
 \begin{thm}[{Main Theorem}]\label{Main Theorem}
  Let $P(x)\in\mathbb{Z}[x]$ be an irreducible monic polynomial whose roots are either real or of modulus $1$.
  Assume at least one root has modulus $1$.
  Let $g$ be the degree of the polynomial $P(x)\in\mathbb{Z}[x]$ and $z _1,z _2,\ldots,z_g$ be the roots of $P(x)$ which are ordered as $\abs{z_1}\geq\abs{z_2}\geq\cdots\geq\abs{z_g}$.\par
  Then, there is a simple abelian variety $X$ of dimension $g$ and an automorphism $f\colon X\longrightarrow X$ with dynamical degrees
  \begin{center}
   $\lambda_0(f)=1, \lambda_k(f)=\prod_{i=1}^{k} \abs{z_i}^2\,(1\leq k\leq g)$.
  \end{center}
 \end{thm}
 \begin{remark}
  If $g=1$, then either $P(x)=x+1$ or $P(x)=x-1$ holds and there exists an $1$-dimensional simple abelian variety $X$ with the identity map $\id_X$ which satisfies $\lambda_0(\id_X)=\lambda_1(\id_X)=1$.
  Thus, the case $g\geq2$ is worth considering.\par
  For the case $g\geq2$, let $\gamma\neq1,-1$ be a root of $P(x)$ whose absolute value is $1$.
  Then $\gamma^{-1}=\overline{\gamma}$ is also a root of $P(x)$.
  This implies that the minimal polynomial $P(x)=a_gx^g+a_{g-1}x^{g-1}+\cdots+a_1x+a_0\,(a_g=1)$ satsfies $a_i=a_{g-i}\,(0\leq i\leq g)$.
  In particular, $a_0=1$ and thus $\lambda_g(f)=1$ and this is compatible with the automorphicity.
 \end{remark}
 For proving Theorem \ref{Main Theorem}, assume $g\geq2$ and take $\gamma$ as in the above remark.
 Then, $K=\mathbb{Q}(\gamma)$ be a quadratic extension of the totally real number field $F=\mathbb{Q}(\gamma+\gamma^{-1})$ with $[F\colon \mathbb{Q}]=\frac{g}{2}$.
 \begin{lemma}[{\cite[Lemma 2.11]{DH22}}]\label{quadratic extension}
  There exists an algebraic integer $a\in\mathcal{O}_F$ such that $F(\sqrt{a})=K$.
 \end{lemma}
 \begin{remark}
  $K=F\left(\gamma-\frac{1}{\gamma}\right)$ and so $a$ can be taken as $a=\left(\gamma-\frac{1}{\gamma}\right)^2$ and we adopt this notation.
 \end{remark}
 On this condition, the next theorem can be applied.
 \begin{thm}[{\cite[Theorem 2.5]{DH22}}]\label{divisional}
  Let $F$ be a totally real number field and $K=F(\sqrt{a})$ a quadratic extension for $a\in\mathcal{O}_F$.
  Then there exists a prime number $p$ such that the quaternion algebra $B=\left(\frac{a,p}{F}\right)$ is divisional.
 \end{thm}
 Since $p\in\mathbb{Z}_{>0}$, $B=\left(\frac{a,p}{F}\right)$ is totally indefinite.
 In order to apply Lemma \ref{Construction} to $\left(\frac{a,p}{F}\right)$, we should construct the order of $\left(\frac{a,p}{F}\right)$. 
 Take an $F$-basis $1,i,j,ij$ of $\left(\frac{a,p}{F}\right)$ with $i^2=a$, $j^2=p$ and $ij=-ji$.
 By the embedding $K=F(\sqrt{a})\hookrightarrow \left(\frac{a,p}{F}\right)$, there is a $\mathbb{Z}$-order $\mathcal{O}:=\mathcal{O}_K\oplus\mathcal{O}_K j$ in $\left(\frac{a,p}{F}\right)$ since $\mathcal{O}_K$ is a finitely generated $\mathbb{Z}$-submodule of $K$ which spans $K$ over $\mathbb{Q}$ and $j^2=p$ gives $\mathcal{O}$ is closed under the multiplication.
 Also, the next lemma holds.
 \begin{lemma}\label{construct positive anti-involution}
  There exists a positive anti-involution on $B=\left(\frac{a,p}{F}\right)$ over $\mathbb{Q}$.
 \end{lemma}
 \begin{remark}
  In \cite{DH22}, this part is explained in the proof of \cite[Lemma 2.12]{DH22}, but we could not understand this part, so we add the proof here.
 \end{remark}
 \begin{proof}[Proof of Lemma \ref{construct positive anti-involution}]
  By Theorem \ref{construction of positive anti-involution}, it suffices to show that there exists $c\in B\backslash F$ such that $c^2\in F$ is totally negative.
  For searching $c$, denote $c=xi+yj+zij$ ($x,y,z\in F$) where $1,i,j,ij$ is an $F$-basis of $B$.
  By calculating, $c^2=x^2a+y^2p-z^2ap\in F$ and substitute
  \begin{align*}
   x=pX\left(\gamma^n+\frac{1}{\gamma^n}\right),\ y=Y\left(\gamma^{n+1}+\frac{1}{\gamma^{n+1}}\right),\ z=Z\left(\frac{1}{\gamma}\right)^n\sum_{i=0}^n\gamma^{2i}
  \end{align*}
  where $X,Y,Z\in\mathbb{Z}_{>0}$ and $n\in\mathbb{Z}_{>0}$.
  Then,
  \begin{align*}
   &c^2=x^2a+y^2p-z^2ap\\
   &=p^2X^2\left(\gamma^n+\frac{1}{\gamma^n}\right)^2\left(\gamma-\frac{1}{\gamma}\right)^2+pY^2\left(\gamma^{n+1}+\frac{1}{\gamma^{n+1}}\right)^2-pZ^2\left(\left(\frac{1}{\gamma}\right)^n\sum_{i=0}^n\gamma^{2i}\right)^2\left(\gamma-\frac{1}{\gamma}\right)^2\\
   &=p^2X^2\left(\gamma^{2n+2}+\frac{1}{\gamma^{2n+2}}-2\left(\gamma^{2n}+\frac{1}{\gamma^{2n}}\right)+\left(\gamma^{2n-2}+\frac{1}{\gamma^{2n-2}}\right)+2\left(\gamma^2+\frac{1}{\gamma^2}\right)-4\right)\\
   &\qquad+pY^2\left(\gamma^{2n+2}+\frac{1}{\gamma^{2n+2}}+2\right)-pZ^2\left(\gamma^{2n+2}+\frac{1}{\gamma^{2n+2}}-2\right).
  \end{align*}
  Thus,
  \begin{align*}
   \frac{c^2}{p}=&(X^2p+Y^2-Z^2)\left(\gamma^{2n+2}+\frac{1}{\gamma^{2n+2}}\right)\\
   &\quad-X^2p\left(2\left(\gamma^{2n}+\frac{1}{\gamma^{2n}}\right)-\left(\gamma^{2n-2}+\frac{1}{\gamma^{2n-2}}\right)-2\left(\gamma^2+\frac{1}{\gamma^2}\right)\right)\\
   &\quad\quad\quad+(-4X^2p+2Y^2+2Z^2)
  \end{align*}
  and a conjugate of $\frac{c^2}{p}$ can be written as
  \begin{align*}
   &(X^2p+Y^2-Z^2)\left(\tau^{2n+2}+\frac{1}{\tau^{2n+2}}\right)\\
   &-X^2p\left(2\left(\tau^{2n}+\frac{1}{\tau^{2n}}\right)-\left(\tau^{2n-2}+\frac{1}{\tau^{2n-2}}\right)-2\left(\tau^2+\frac{1}{\tau^2}\right)\right)+(-4X^2p+2Y^2+2Z^2)\tag*{($\ast$)}
  \end{align*}
  where $\tau$ is a conjugate of $\gamma$. Define
  \begin{align*}
   &S:=\{\text{conjugates of }\gamma\text{ whose modulus is }1\}\\
   &T:=\{\text{conjugates of }\gamma\text{ which is real}\}
  \end{align*}
  with $S\cap T=\emptyset$ since $P(x)\neq x+1,x-1$.\par
  For any real number $\tau\neq\pm1$,
  \begin{align*}
   2\left(\tau^{2N}+\frac{1}{\tau^{2N}}\right)-\left(\tau^{2N-2}+\frac{1}{\tau^{2N-2}}\right)-2\left(\tau^2+\frac{1}{\tau^2}\right)
  \end{align*}
  diverges to $+\infty$ as $N\rightarrow+\infty$.
  Thus, there exists a sufficiently large $N_0\in\mathbb{Z}_{>0}$ such that for any integer $N\geq N_0$,
  \begin{align*}
   2\left(\tau^{2N}+\frac{1}{\tau^{2N}}\right)-\left(\tau^{2N-2}+\frac{1}{\tau^{2N-2}}\right)-2\left(\tau^2+\frac{1}{\tau^2}\right)>-2
  \end{align*}
  holds for all $\tau\in T$.
  Define 
  \begin{align*}
   \alpha:=\underset{\tau\in S}{\mathrm{max}}\ \left(\tau^2+\frac{1}{\tau^2}\right).
  \end{align*}
  Now, by $-1,1\notin S$, $\alpha<2$.
  By applying Lemma \ref{analogous to Kronecker's Density Theorem} which is proved below, for any small neighborhood of $1$, there are infinitely many $n$ such that all $\tau^n$ ($\tau\in T$) can be inside the neighborhood.
  Thus, there exists $N\geq N_0$ such that
  \begin{align*}
   2\left(\tau^{2N}+\frac{1}{\tau^{2N}}\right)>2\alpha
  \end{align*}
  and then
  \begin{align*}
   2\left(\tau^{2N}+\frac{1}{\tau^{2N}}\right)-\left(\tau^{2N-2}+\frac{1}{\tau^{2N-2}}\right)-2\left(\tau^2+\frac{1}{\tau^2}\right)>2\alpha-2-2\alpha=-2
  \end{align*}
  for all $\tau\in S$.
  Thus, there exists $\epsilon>0$ such that
  \begin{align*}
   2\left(\tau^{2N}+\frac{1}{\tau^{2N}}\right)-\left(\tau^{2N-2}+\frac{1}{\tau^{2N-2}}\right)-2\left(\tau^2+\frac{1}{\tau^2}\right)>-2+\epsilon
  \end{align*}
  for all $\tau\in S\cup T$.
  For this $N$, ($\ast$) is less than
  \begin{align*}
   (X^2p+Y^2-Z^2)\left(\tau^{2N+2}+\frac{1}{\tau^{2N+2}}\right)+(2-\epsilon)X^2p+(-4X^2p+2Y^2+2Z^2)\\
   =(X^2p+Y^2-Z^2)\left(\tau^{2N+2}+\frac{1}{\tau^{2N+2}}\right)+(-(2+\epsilon)X^2p+2Y^2+2Z^2).\tag*{($\ast\ast$)}
  \end{align*}
  By substituting $Y=1$, there exist infinitely many solutions for the equation
  \begin{align*}
   1=Z^2-X^2p\quad(X,Z\in\mathbb{Z})
  \end{align*}
  since this is a Pell's equation.
  By substituting a solution for this equation,
  \begin{align*}
   (\ast\ast)=-(2+\epsilon)X^2p+2Y^2+2Z^2=-\epsilon X^2p+4
  \end{align*}
  and as $\abs{X}\to\infty$, ($\ast\ast$) can be nagative for all $\tau\in S\cup T$ and so for these $X,Y,Z$ and $N$, $c^2\in F$ is totally negative.
  Now the proof is concluded by proving the next lemma.
 \begin{lemma}\label{analogous to Kronecker's Density Theorem}
  Let $M$ be a finite set of complex numbers whose modulus is $1$ and let $\epsilon>0$ be an arbitrary small positive number.\par
  Then, there are infinitely many $n\in\mathbb{Z}_{>0}$ such that $\abs{1-z^n}<\epsilon$ for all $z\in M$.
 \end{lemma}
 This lemma is reduced to the next lemma via the isomorphism
 \begin{align*}
  \mathbb{R}/\mathbb{Z}\simeq S^1:=\{z\in\mathbb{C}\mid\abs{z}=1\}.
 \end{align*}
 \begin{lemma}
  Take $r_1,r_2,\ldots,r_m\in\mathbb{R}$ and fix $\epsilon>0$ arbitrary.\par
  Then, there are infinitely many $n\in\mathbb{Z}_{>0}$ such that either $\{nr_i\}<\epsilon$ or $\{nr_i\}>1-\epsilon$ holds for each $1\leq i\leq m$.
  Here, $\{x\}$ represents the decimal part of $x\in\mathbb{R}$ and so $\{x\}=x-\lfloor x\rfloor$.
 \end{lemma}
 \begin{proof}
  Define
  \begin{align*}
   A:=\{(\{nr_1\},\{nr_2\},\ldots,\{nr_m\})\in[0,1)^m\}_{n\in\mathbb{Z}_{>0}}
  \end{align*}
  as a subset of $[0,1)^m$.
  Let $N\in\mathbb{Z}_{>0}$ be an integer which satisfies $\frac{1}{N}<\epsilon$.
  Cut the hypercube $[0,1)^m$ into $N^m$ pieces of hypercubes whose length of a side is $\frac{1}{N}$.\par
  Then by applying the pigeonhole principle, there exist $k,k'\in\mathbb{Z}_{>0}$ ($k<k'$) such that $(\{kr_1\},\{kr_2\},\ldots,\{kr_m\})$ and $(\{k'r_1\},\{k'r_2\},\ldots,\{k'r_m\})$ are contained in a common small hypercube.
  By the definition of the hypercubes, $n=k'-k$ satisfies the condition in the lemma.\par
  Moreover, for any large positive integer $n_0$, define
  \begin{align*}
   A_{n_0}:=\{(\{nr_1\},\{nr_2\},\ldots,\{nr_m\})\in[0,1)^m\}_{n=1,n_0+1,2n_0+1,\ldots}
  \end{align*}
  and by the same deduction as above, there exists $n\geq n_0$ which satisfies the condition.
  Thus, the proof is concluded. 
 \end{proof}

 \end{proof}

 Hence the assumption of Lemma \ref{Construction} holds for $B=\left(\frac{a,p}{F}\right)$ and its order $\mathcal{O}$, and therefore there is a $g$-dimensional simple abelian variety $X$ such that $\mathcal{O}$ is embedded into $\mathrm{End}(X)$.
 This simple abelian variety is of Type 2 in Table \ref{table}.
 \begin{proof}[Proof of Theorem \ref{Main Theorem}]
  $\gamma$ and $\overline{\gamma}=\frac{1}{\gamma}$ are elements of $\mathcal{O}_K$ and hence they are in $\mathcal{O}$, so $\gamma\in \mathrm{End}(X)$ is an automorphism of $X$.\par
  Let $\rho_1,\rho_2,\ldots,\rho_g$ be the eigenvalues of $\rho_a(\gamma)$.
  By applying Section \ref{Calculation} for the $g$-dimensional simple abelian variety $X$, $B=\mathrm{End}_\mathbb{Q}(X)=\left(\frac{a,p}{F}\right)$, $\alpha=\gamma$, $d=d'=2$, $d''=g$, $e=\frac{g}{2}$ and $F(x)=P(x)$,
  \begin{equation*}
   \prod_{i=1}^{g}(x-\rho_{i})(x-\overline{\rho_{i}})=P(x)^2=\prod_{i=1}^{g}(x-z_i)^2
  \end{equation*}
  as polynomials in $x$.\par
  Thus, $\rho_1, \ldots, \rho_g,\overline{\rho_1}, \ldots, \overline{\rho_g}$ is a permutation of $z_1, \ldots, z_g, z_1, \ldots, z_g$.\par
  Therefore, we may assume that $\abs{\rho_i}=\abs{z_i}$ ($1\leq i\leq g$) and this implies that $\lambda_0(\gamma)=1$, $\lambda_k(\gamma)=\prod_{i=1}^{k} \abs{z_i}^2$ ($1\leq k\leq g$).
 \end{proof}
 The next similar theorem can be proved in an analogous way.
 \begin{theorem}\label{Main Theorem1}
  Let $P(x)\in\mathbb{Z}[x]$ be an irreducible monic polynomial whose constant term is $\pm 1$ and whose roots are all real.
  Let $g$ be the degree of the polynomial $P(x)\in\mathbb{Z}[x]$ and $z _1,z _2,\ldots,z_g\in\mathbb{R}$ be the roots of $P(x)$ which are ordered as $\abs{z_1}\geq\abs{z_2}\geq\cdots\geq\abs{z_g}$.\par
  Then, there is a simple abelian variety $X$ of dimension $g$ and an automorphism $f\colon X\longrightarrow X$ with dynamical degrees
  \begin{center}
   $\lambda_0(f)=1, \lambda_k(f)=\prod_{i=1}^{k} \abs{z_i}^2\,(1\leq k\leq g)$.
  \end{center}
 \end{theorem}
 \begin{proof}
  Let $\delta$ be a root of $P(x)$ and so $\delta$ is an algebraic integer.
  The conjugates of $\delta$ are all real, and so $F:=\mathbb{Q}(\delta)$ is a totally real number field.
  The identity map on $F$ is a positive anti-involution and so by Lemma \ref{Construction1} with $m=1$, there is a $g$-dimensional simple abelian variety $X$ such that $\mathrm{End}(X)$ contains $\mathcal{O}_F$.
  This simple abelian variety is of Type 1 in Table \ref{table}.\par
  Since the constant term of $P(x)$ is an unit in $\mathbb{Z}$, $\delta^{-1}$ is also an integral element, and so $\delta, \delta^{-1}$ are in $\mathcal{O}_F$.
  Thus, $\delta, \delta^{-1}$ can be regarded as endomorphisms of the simple abelian variety $X$ and so these are automorphisms.
  Let $\rho_1,\rho_2,\ldots,\rho_g$ be the eigenvalues of $\rho_a(\delta)$.
  By applying Section \ref{Calculation} for the $g$-dimensional simple abelian variety $X$, $K=F$, $\alpha=\delta$, $d=1$, $e=g$, $d'=g$ and $F(x)=P(x)$,
  \begin{align*}
   \prod_{i=1}^{g}(x-\rho_{i})(x-\overline{\rho_{i}})=P(x)^2=\prod_{i=1}^{g}(x-z_i)^2
  \end{align*}
  as polynomials in $x$.\par
  Thus, $\rho_1, \ldots, \rho_g,\overline{\rho_1}, \ldots, \overline{\rho_g}$ is a permutation of $z_1, \ldots, z_g, z_1, \ldots, z_g$.\par
  Therefore, we may assume that $\abs{\rho_i}=\abs{z_i}$ ($1\leq i\leq g$) and this implies that $\lambda_0(\delta)=1$, $\lambda_k(\delta)=\prod_{i=1}^{k} \abs{z_i}^2$ ($1\leq k\leq g$).
 \end{proof}

\section{Corollaries}\label{Corollaries}
 This section is devoted to corollaries of Theorem \ref{Main Theorem}.
 \begin{lemma}\label{Main lemma}
  For a Salem number $\lambda>0$ and a positive integer $n\in\mathbb{Z}$, $\sqrt[n]{\lambda}$ has an algebraic conjugate of modulus $1$.
 \end{lemma}
 \begin{proof}
  Let $Q(x)\in\mathbb{Z}[x]$ be the minimal polynomial of $\mu=\sqrt[n]{\lambda}$.\par
  The polynomial $R(x)=Q(x)Q(\zeta_n x)\cdots Q(\zeta_n^{n-1}x)$ is invariant under $\mathrm{Gal}(\mathbb{Q}(\zeta_n)/\mathbb{Q})$, so $R(x)$ has coefficients in $\mathbb{Q}$ and has non-zero coefficients only at $(an)$-th degrees ($a\in\mathbb{Z}_{>0}$) since $R(x)=R(\zeta_n x)$.
  Moreover, all the coefficients of $R(x)$ are integers since they are in $\mathbb{Z}[\zeta_n]$ and so they are algebraic integers in $\mathbb{Q}$.\par
  Thus, $S(x)=R(x^{\frac{1}{n}})$ is a polynomial in $\mathbb{Z}[x]$ and $\lambda$ is a root of this polynomial.
  Therefore, $S(x)$ is divided by the Salem polynomial of $\lambda$ and so $S(x)$ has a root of modulus $1$.
  This implies that $Q(x)$ has a root of modulus $1$ and the lemma is concluded.
 \end{proof}
 \begin{corollary}\label{realize salem}
  Any Salem number is realized as the first dynamical degree of an automorphism of a simple abelian variety over $\mathbb{C}$.
 \end{corollary}
 \begin{proof}
  Let $\lambda$ be a Salem number, $P(x)$ be the minimal polynomial of $\lambda$ and $Q(x)$ be the minimal polynomial of $\sqrt{\lambda}$.\par
  By Lemma \ref{Main lemma} for the case $n=2$, $\sqrt{\lambda}$ has a conjugate of modulus $1$.
  $Q(x)$ is a factor of $P(x^2)\in\mathbb{Z}[x]$ and since the roots of $P(x^2)$ are either real or of modulus $1$, so are the roots of $Q(x)$.
  Thus, the minimal polynomial of $\sqrt{\lambda}$ satisfies the conditions in Theorem \ref{Main Theorem} and there exists an automorphism of a simple abelian variety for this minimal polynomial.
  Now the element $\sqrt{\lambda}$ has the maximal absolute value among the roots of the minimal polynomial and therefore the first dynamical degree of the automorphism is $\sqrt{\lambda}^2=\lambda$.
 \end{proof}
 The next lemma is used in the proof of Corollary \ref{realize Salem} for analyzing the minimal polynomial of $\sqrt{\lambda}$.
 \begin{lemma}[{Kronecker's theorem}]
  Let $P(x)\in\mathbb{Z}[x]$ be a monic polynomial whose roots are all of modulus $1$.
  Then, the roots of $P(x)$ are all roots of unity.
 \end{lemma}
 \begin{proof}
  Let $z_1, \ldots, z_n$ be the roots of $P(x)$ and write
  \begin{align*}
   P(x)=\prod_{l=1}^{n}(x-z_l).
  \end{align*}
  Define  $P_k(x)$ as
  \begin{align*}
   P_k(x)=\prod_{l=1}^{n}(x-{z_l}^k)
  \end{align*}
  for all $k\in\mathbb{Z}_{>0}$. 
  Now the monic polynomial $P_k(x)$ is of degree $n$ and its $i$-th coefficient is an integer whose absolute value is at most $\binom{n}{i}$ by the triangle inequality, and so the number of the candidates of the coefficients of $P_k(x)$ is finite.\par
  Thus, $\{P_k(x)\}_{k=1,2,\ldots}$ is a finite set of polynomials and there are some $i, j\in\mathbb{Z}_{>0}$ ($i\neq j$) such that $P_i(x)=P_j(x)$.\par
  Therefore, $\{{z_l}^i\}_{1\leq l\leq n}=\{{z_l}^j\}_{1\leq l\leq n}$ and for each $l$, there are relations ${z_l}^i={z_{l^{\prime}}}^j$, ${z_{l^{\prime}}}^i={z_{l^{\prime\prime}}}^j, \ldots$ for some $l',l'',\ldots$.
  This implies that ${z_l}^{i^m}={z_l}^{j^m}$ for some $m>0$ and since $i\neq j$, $z_l$ is a root of unity for each $l$. 
 \end{proof}
 \begin{corollary}\label{realize Salem}
  For a Salem number $\lambda$ of degree $g$, there is an automorphism of a simple abelian variety, whose dimension is $g$ or $2g$, such that the first dynamical degree is equal to $\lambda$.
 \end{corollary}
 \begin{proof}
  Let $P(x)$ be the minimal polynomial of $\lambda$.
  As in Definition \ref{Definition of Salem}, $P(x)$ has $\lambda$, $\frac{1}{\lambda}$ as its roots and all other roots are of modulus $1$. 
  Now the minimal polynomial $Q(x)$ of $\sqrt{\lambda}$ is a factor of $P(x^2)$ and by Lemma \ref{Main lemma}, $Q(x)$ has a root of modulus $1$.
  Thus, the set of the roots of $Q(x)$ is invariant under the map $z\mapsto \frac{1}{z}$.
  By assuming that $P(x^2)$ is reducible, the set of the roots of $Q(x)$ is one of the next.
  \begin{enumerate}
  \renewcommand{\labelenumi}{\arabic{enumi}.}
   \item\label{(1)} $\sqrt{\lambda}, \frac{1}{\sqrt{\lambda}}, z_1, \overline{z_1}, \ldots, z_n, \overline{z_n}$\quad (where $z_i$ is of modulus $1$ for all $i$)
   \item\label{(2)} $\sqrt{\lambda}, \frac{1}{\sqrt{\lambda}}, -\sqrt{\lambda}, -\frac{1}{\sqrt{\lambda}}, z_1, \overline{z_1}, \ldots, z_n, \overline{z_n}$\quad (where $z_i$ is of modulus $1$ for all $i$)
  \end{enumerate}\par
  On the case \ref{(2)}, the roots of $\frac{P(x^2)}{Q(x)}$ are all of modulus $1$, so by Kronecker's theorem, they are roots of unity, a contradiction.\par
  On the case \ref{(1)}, $Q(-x)$ is also a factor of $P(x^2)$ and $Q(x)$, $Q(-x)$ have no common roots.
  Moreover, $\frac{P(x^2)}{Q(x)Q(-x)}$ is a constant or a polynomial whose roots are all of modulus $1$. 
  The latter contradicts as above and so $P(x^2)=Q(x)Q(-x)$ holds by considering the leading coefficient.\par
  Thus either $Q(x)=P(x^2)$ or $Q(x)Q(-x)=P(x^2)$ holds for $Q(x)$.
  This implies that the minimal polynomial of $\sqrt{\lambda}$ has degree $g$ or degree $2g$ and this corresponds to the dimension of the simple abelian variety (cf.\ Theorem \ref{Main Theorem}).
 \end{proof}
 \begin{rem}
  By combining with Lemma \ref{Construction2}, any Salem number of degree $g$ is realized as the first dynamical degree of an automorphism of a $2g$-dimensional simple abelian variety.\par 
  By observing the proof of Corollary \ref{realize Salem}, the degree of the minimal polynomial of $\sqrt{\lambda}$ is $g$ if and only if $\lambda$ is the square of some Salem numbers.\par
  Thus, it can be said that a Salem number $\lambda$ of degree $g$ is realized as the first dynamical degree of an automorphism of a $g$-dimensional simple abelian variety when $\lambda$ is the square of some Salem number.
 \end{rem}
 Let $\lambda$ be a Salem number of degree $g$.\par
 If $\lambda$ is the square of some Salem number, the automorphism $f$ of a $g$-dimensional simple abelian variety constructed in Corollary \ref{realize Salem} satisfies
 \begin{align*}
  \lambda_0(f)=\lambda_g(f)=1, \lambda_i(f)=\lambda\ (1\leq i\leq g-1)
 \end{align*}
 as in Theorem \ref{DH22 theorem}.
 If $\lambda$ is not the square of any Salem number, the automorphism $f$ of a $2g$-dimensional simple abelian variety constructed in Corollary \ref{realize Salem} satisfies
 \begin{align*}
  \lambda_0(f)=\lambda_{2g}(f)=1, \lambda_1(f)=\lambda_{2g-1}(f)=\lambda, \lambda_i(f)=\lambda^2\ (2\leq i\leq 2g-2).
 \end{align*}\par
 By summarizing Theorem \ref{Main Theorem} and Theorem \ref{Main Theorem1}, the next corollary is implied.
 \begin{corollary}\label{possible dynamical degrees}
  Let $\lambda$ be an algebraic integer of degree $g$ whose conjugates are either real or of modulus $1$, and such that $\frac{1}{\lambda}$ is also an algebraic integer.\par
  Then there is an automorphism of a $g$-dimensional simple abelian variety which corresponds to $\lambda$ and the first dynamical degree of this automorphism is the square of the maximal absolute value of the conjugates of $\lambda$.
 \end{corollary}
 \begin{remark}\label{possible dynamical degrees1}
  By combining the corollary with Lemma \ref{Construction1} and Lemma \ref{Construction2}, there is an automorphism, which corresponds to $\lambda$, of a $gm$-dimensional simple abelian variety for any $m\in\mathbb{Z}_{>0}$.
 \end{remark}

\section{Small first dynamical degrees}\label{Small first dynamical degree}
 In \cite{DH22}, there are analyses of the size relationship between the dynamical degrees of an automorphism of a simple abelian variety.\par
 From here, we consider what values appear as the first dynamical degrees of automorphisms of simple abelian varieties.\par
 This section is devoted to examining small first dynamical degrees of automorphisms of simple abelian varieties with some restrictions.
 \subsection{Small first dynamical degrees with restricting the dimension}\label{Small dynamical degree with restricted dimension}
  Fix a positive integer $g\geq2$ and let $X$ be an abelian variety whose dimension is $g$.\par
  Denote $X=\mathbb{C}^g/\Lambda$ and then the singular homology is calculated as
  \begin{align*}
   \mathrm{H}_1(X)\simeq\pi^{-1}(0)\simeq\Lambda\simeq\mathbb{Z}^{2g}
  \end{align*}
  with the universal covering $\pi:\mathbb{C}^g\rightarrow X$.\par
  Let $f$ be a surjective endomorphism of $X$.
  As in Section \ref{Dynamical degree}, the eigenvalues of $\rho_r(f)$ can be written as $\rho_1,\cdots,\rho_g,\overline{\rho_1},\ldots,\overline{\rho_g}$ and they are also the eigenvalues of the isomorphism $f_*:\mathrm{H}_1(X)\longrightarrow \mathrm{H}_1(X)$ with multiplicity.
  The characteristic polynomial of the linear transformation $f_*$ is an equation with coefficients in $\mathbb{Z}$ and can be written as $p(x)=x^{2g}+a_1x^{2g-1}+\cdots+a_{2g}\in\mathbb{Z}[x]$.\par
  Let $c>1$ be a real number.
  Assume the first dynamical degree of $f$ is less than $c^2$ and then all the roots of $p(x)$ have absolute values less than $c$.
  This restriction gives
  \begin{align*}
   \abs{a_i}<c^i\binom{2g}{i}\quad(1\leq i\leq 2g)
  \end{align*}
  and so the number of the candidates of the minimal polynomials is finite.
  Thus, the number of the candidates of the first dynamical degrees less than $c^2$ is finite.\par
  Set a real number $c>1$ such that there is a surjective endomorphism of a $g$-dimensional abelian variety whose first dynamical degree is not $1$ and smaller than $c^2$.
  Then the set of the first dynamical degrees inside the open interval $(1,c^2)$ can be not empty.
  Thus, there is the minimum value of the first dynamical degrees besides $1$ and so the next theorem holds.
  \begin{theorem}
   Define
   \begin{align*}
    \mathcal{A}_g:=\left\{
                    \begin{array}{l}
                     \text{first dynamical degrees of surjective endomorphisms}\\
                     \text{  of abelian varieties over $\mathbb{C}$ whose dimension is $g$}
                    \end{array}
                   \right\}\backslash\{1\}.
   \end{align*}
   for an integer $g\geq 2$.
   This set has the minimum value.
  \end{theorem}
  By the similar deduction, the next theorem also holds.
  \begin{theorem}
   For an integer $g\geq2$,
   \begin{align*}
    \mathcal{B}_g:=\left\{
                    \begin{array}{l}
                     \text{first dynamical degrees of automorphisms}\\
                     \text{  of simple abelian varieties over $\mathbb{C}$ whose dimension is $g$}
                    \end{array}
                   \right\}\backslash\{1\}
  \end{align*}
  has the minimum value.
  \end{theorem}
  \begin{remark}
   Fix an integer $g\geq2$.
   The smallest value of first dynamical degrees of automorphisms of $g$-dimensional simple abelian varieties larger than $1$ can be determined in lower dimensions.\par
   For example, if $g=2$, then the smallest first dynamical degree larger than $1$ is $4\mathrm{cos}^2\left(\frac{\pi}{5}\right)=\left(\frac{1+\sqrt{5}}{2}\right)^2=\frac{3+\sqrt{5}}{2}=2.6180\cdots$.\par
   If $g=3$, then the smallest first dynamical degree larger than $1$ is $4\mathrm{cos}^2\left(\frac{\pi}{7}\right)=3.2469\cdots$.\par
   If $g=4$, then the smallest first dynamical degree larger than $1$ is $2\mathrm{cos}\left(\frac{\pi}{5}\right)=\frac{1+\sqrt{5}}{2}=1.6180\cdots$.\par
   If $g=5$, then the smallest first dynamical degree larger than $1$ is $4\mathrm{cos}^2\left(\frac{\pi}{11}\right)=3.6825\cdots$.\par
  \end{remark}

 \subsection{Possible small first dynamical degrees}\label{possible}
  From here, we consider the possible values of the first dynamical degrees of automorphisms of simple abelian varieties which are close to $1$, in each type.
  For a simple abelian variety $X$, $B=\mathrm{End}_\mathbb{Q}(X)$ and $K$ are as in Section \ref{Endomorphism algebra of simple abelian varieties}.
  The next lemma is used later.
  \begin{lemma}\label{cyclotomic}
   Let $P(x)=x^n+a_1x^{n-1}+\cdots+a_n\in\mathbb{Z}[x]$ be an irreducible monic polynomial of degree $n\geq2$ which has only real roots and denote the maximal absolute value of the roots of $P(x)$ by $\mu$.
   If $\mu<2$, then $P(x)$ is a polynomial which satisfies $x^nP\left(x+\frac{1}{x}\right)=\Phi_N(x)$ for some $N\in\mathbb{Z}_{>0}$ where $\Phi_N(x)$ is the $N$-th cyclotomic polynomial.
   Also, the roots of $P(x)$ can be written as $2\mathrm{cos}\left(\frac{2m\pi}{N}\right)$ for some $m\in\mathbb{N}$.\par
   Under the notation, the maximal absolute value of the roots of $P(x)$ can be written as
   \begin{align*}
    \left\{
     \begin{array}{ll}
      2\mathrm{cos}\left(\frac{2\pi}{N}\right) & (N:\text{even})\\ [+5pt]
      2\mathrm{cos}\left(\frac{\pi}{N}\right) & (N:\text{odd})
     \end{array}.
    \right.
   \end{align*}
  \end{lemma}
  \begin{proof}
   By the condition $\mu<2$, the monic polynomial $Q(x)=x^nP(x+\frac{1}{x})\in\mathbb{Z}[x]$ has only roots of modulus $1$.
   Thus, by Kronecker's theorem, $Q(x)$ is a product of cyclotomic polynomials and by the irreducibility of $P(x)$, $Q(x)$ is irreducible.
   Therefore, $Q(x)$ is a cyclotomic polynomial and can be written as $Q(x)=\Phi_N(x)$ for some $N\in\mathbb{Z}_{>0}$.\par
   Then the roots of $P(x)$ can be written as $\zeta_N^m+\frac{1}{\zeta_N^m}$ where  $m\in\mathbb{Z}$ is relatively prime to $N$.\par
   Thus, the roots of $P(x)$ can be rewritten as\\
   \begin{align*}
    \zeta_N^m+\frac{1}{\zeta_N^m}=2\mathrm{cos}\left(\frac{2m\pi}{N}\right).
   \end{align*}
   and also the maximal absolute value of the roots of $P(x)$ is
   \begin{align*}
    \left\{
     \begin{array}{ll}
      2\mathrm{cos}\left(\frac{2\pi}{N}\right) & (N:\text{even})\\ [+5pt]
      2\mathrm{cos}\left(\frac{\pi}{N}\right) & (N:\text{odd})
     \end{array}.
    \right.
   \end{align*}
  \end{proof}

  \begin{flushleft}{\bf{Type 1}}\end{flushleft}\par
   Let $\alpha$ be a root of an irreducible polynomial of the form $P(x)=x^m+a_1x^{m-1}+\cdots\pm 1\in\mathbb{Z}[x]$ which has only real roots.\par
   These cover all automorphisms of simple abelian varieties of Type 1.\par
   As in Theorem \ref{Main Theorem1}, any $\alpha$ which is the root of such a polynomial corresponds to an automorphism of some $nm$-dimensional simple abelian variety ($n\in\mathbb{Z}_{>0}$) by Lemma \ref{Construction1}.\par
   For searching a small first dynamical degree (it is equal to the square of the maximal absolute value of the roots of $P(x)$ as in Section \ref{Calculation}), assume $P(x)$ has only roots which have absolute values less than $2$.\par
   If $m=1$, then the square of the maximal absolute value of the roots of $P(x)$ is $1$.\par
   If $m\geq2$, by Lemma \ref{cyclotomic}, the maximal absolute value of the roots of $P(x)$ is
   \begin{align*}
    \left\{
     \begin{array}{ll}
      2\mathrm{cos}\left(\frac{2\pi}{N}\right) & (N:\text{even})\\ [+5pt]
      2\mathrm{cos}\left(\frac{\pi}{N}\right) & (N:\text{odd})
     \end{array}.
    \right.
   \end{align*}
   Since the constant term of $P(x)$ is $\pm1$,
   by comparing with Theorem \ref{constant term of minimal polynomial}, the possible values of $N$ is $N=3,5,6,7,9,10,11,\ldots$.
   Thus, the first dynamical degrees of this type can be sorted in ascending order as
   \begin{align*}
    1,\ 4\mathrm{cos}^2\left(\frac{\pi}{5}\right),\ 4\mathrm{cos}^2\left(\frac{\pi}{7}\right),\ 4\mathrm{cos}^2\left(\frac{\pi}{9}\right),\ 4\mathrm{cos}^2\left(\frac{\pi}{11}\right).\ 4\mathrm{cos}^2\left(\frac{\pi}{12}\right),\ldots
   \end{align*}
  \begin{flushleft}{\bf{Type 2, Type 3}}\end{flushleft}\par
   Let $K$ be a totally real number field with $[K:\mathbb{Q}]=m$.
   Also, let $\alpha$ be an element of $U_K$ or be an element of $B$ which has a minimal polynomial of the form $x^2+sx+t$ ($s\in\mathcal{O}_K$, $t\in U_K$) over $K$.
   These cover all automorphisms of simple abelian varieties of Type 2 and Type 3.\par
   Without considering the realizability, the possible first dynamical degrees of automorphisms for the former case is the same as in Type 1.\par
   For considering the latter case, as in Section \ref{Calculation}, the first dynamical degree of an automorphism corresponding to $\alpha$ is calculated as the square of the maximal absolute value of the roots of
   \begin{align*}
    \prod_{i=1}^m (x^2+\sigma_i(s)x+\sigma_i(t))
   \end{align*}
   where $\left\{\sigma_i\right\}_{1\leq i\leq m}$ is the set of all $\mathbb{Q}$-embeddings $K\hookrightarrow\mathbb{C}$.
   Now each $\sigma_i(t)$ is a conjugate element of $t\in U_K$ and these elements are the roots of the common irreducible polynomial of the form $P(x)=x^{m'}+a_1x^{m'-1}+\cdots\pm 1\in\mathbb{Z}[x]$ ($m'\leq m$).
   Thus, one of the absolute value of $\sigma_i(t)$'s is not less than $1$ and so we can assume $\abs{t}\geq1$.\par
   If $\abs{t}>1$, the maximal absolute value of the roots of $x^2+\sigma_i(s)x+\sigma_i(t)$ is not less than $\sqrt{\abs{\sigma_i(t)}}$.
   The lower bound of the maximal absolute value of $\abs{\sigma_i(t)}$'s is $2\mathrm{cos}\left(\frac{\pi}{5}\right)=\frac{1+\sqrt{5}}{2}$ as the deduction in Type 1. 
   Thus, $\frac{1+\sqrt{5}}{2}$ is the lower bound of the first dynamical degrees besides $1$ for this case.\par
   If $\abs{t}=1$, since $K$ is totally real, $t=\pm 1$.\par
   If $t=1$, the maximal absolute value of the roots of all over $x^2+\sigma_i(s)x+1$ is achived by one of the roots of $x^2+\sigma_i(s)x+1$ where $\abs{\sigma_i(s)}$ is maximal.
   Now all $\sigma_i(s)$ are the roots of an identical irreducible polynomial and they are all real.
   Thus, the roots of $x^2+\sigma_i(s)x+1$ are either real or of modulus $1$ and so the possible first dynamical degrees for this case all appear in Corollary \ref{possible dynamical degrees}.\par
   If $t=-1$, the maximal absolute value of the roots of all over $x^2+\sigma_i(s)x-1$ is achieved by one of the roots of $x^2+\sigma_i(s)x-1$ where $\abs{\sigma_i(s)}$ is maximal.
   Assume the maximal absolute value is not $1$ and so $s\neq0$.
   All $\sigma_i(s)$ are the roots of common polynomial whose roots are all real.
   Now $\prod_{i=1}^m \sigma_i(s)$ is an integer and by $\sigma_i(s)\neq0$, $\prod_{i=1}^m \abs{\sigma_i(s)}\geq1$ and so for some $i_0$, $\abs{\sigma_{i_0}(s)}\geq1$.
   The roots of polynomial $x^2+\sigma_{i_0}(s)x-1$ can be denoted by $z_{i_0}$, $-\frac{1}{z_{i_0}}$ ($z_i\in\mathbb{R}$) with $\abs[\big]{z_{i_0}-\frac{1}{z_{i_0}}}=\abs{\sigma_{i_0}(s)}\geq1$.
   Thus, $\mathrm{max}\left\{\abs{z_{i_0}},\frac{1}{\abs{z_{i_0}}}\right\}$ is not less than $\frac{1+\sqrt{5}}{2}$.
   Therefore, the first dynamical degree of $\alpha$ is at least $\left(\frac{1+\sqrt{5}}{2}\right)^2$ besides $1$ for this case.
  \begin{flushleft}{\bf{Type 4}}\end{flushleft}\par
   Let $K$ be a CM-field with $[K:\mathbb{Q}]=2m$ and let $\alpha$ be an element which has a minimal polynomial of the form $P(x)\in\mathcal{O}_K[x]$ whose constant term is in $U_K$ and whose degree is not more than $d$.
   These cover all automorphisms of simple abelian varieties of Type 4.\par
   Without considering the realizability, the first dynamical degree of an automorphism corresponding to $\alpha$ is calculated as the square of the maximal absolute value of the roots of
   \begin{align*}
    \prod_{i=1}^{2m} (\sigma_i(P(x)))=\prod_{j=1}^m (\sigma'_j(P(x)\overline{P(x)}))\in\mathbb{Z}[x]
   \end{align*}
   where $\left\{\sigma_i\right\}_{1\leq i\leq 2m}$ (resp. $\left\{\sigma'_j\right\}_{1\leq j\leq m}$) is the set of all $\mathbb{Q}$-embeddings $K\hookrightarrow\mathbb{C}$ (resp. $K_0\hookrightarrow\mathbb{C}$).\par
   The possible values of small first dynamical degrees of this type would be considered in Section \ref{First dynamical degree close to 1}.

\section{First dynamical degrees close to $1$}\label{First dynamical degree close to 1}
 In this section, we consider the next question.
 \begin{question}
  Without fixing the dimension of simple abelian varieties, then does there exist the minimum value of first dynamical degrees of automorphisms?
 \end{question}
 \noindent Denote
 \begin{align*}
  \Delta=\{\text{first dynamical degrees of automorphisms of simple abelian varieties over $\mathbb{C}$}\}\backslash\{1\}
 \end{align*} 
 for convenience.\par
 This quetion is immediately solved negatively if there does not exist the smallest number of the Salem numbers, because all Salem numbers are contained in $\Delta$ by Corollary \ref{realize salem}.
 But this problem is unsolved yet, and also it is conjectured that there is the smallest Salem number (\cite[Chapter 3]{Bor02}), so another proof is needed.\par
 The above question is solved negatively in Theorem \ref{minimum value of Delta} and reduced to the next theorem.
 
 \begin{theorem}\label{automorphism}
  For a prime number $p>3$ and an integer $1\leq k\leq\frac{p-1}{2}$, there are a CM-field $K$, a central simple division algebra $B$ over $K$ of the form $B=\oplus_{i=0}^{p-1} u^{i}L$ with a symbol $u$, a cyclic extension $L/K$ of degree $p$, and $[B:K]=p^2$ such that $\sqrt[p]{\zeta_p^k+\frac{1}{\zeta_p^k}}\in L$.
 \end{theorem}
 For proving this theorem, the next proposition is useful.
 \begin{proposition}[{cf.\ \cite[Chapter 15.1, Corollary d]{Pie82}}]\label{construct division algebra}
  Let $E/F$ be a cyclic extension of fields (i.e., a Galois extension with the cyclic Galois group) of degree $n$ and denote $G=\mathrm{Gal}(E/F)$ and its generator by $\sigma$.\par
  Let $a\in F^{\times}$ be an element and assume the order of $a$ in $F^{\times}/\mathrm{N}_{E/F}(E^{\times})$ is exactly $n$.
  Take a symbol $u$ as $u^n=a$.\par
  Define $B=\oplus_{i=0}^{n-1} u^{i}E$ and construct the multiplication on $B$ as $e\cdot u=u\cdot\sigma(e)$ for $e\in E$.
  Then $B$ is a division ring and also a central simple algebra over $F$ with $[B:F]=n^2$.
 \end{proposition}
 \begin{remark}\label{remark about the order}
  Assume $a\in\mathcal{O}_F$, then the division ring $B$ constructed in Proposition \ref{construct division algebra} has an order $\mathcal{O}=\oplus_{i=0}^{n-1} u^{i}\mathcal{O}_E$ since $\sigma$ projects $\mathcal{O}_E$ into $\mathcal{O}_E$.
 \end{remark}
 Before proceeding to the proof, some lemmas are required.
 \begin{lemma}\label{CM-field}
  Let $p$ be an odd prime number.\par
  Then, $\mathbb{Q}\left(\zeta_p+\frac{1}{\zeta_p}\right)$ is a totally real number field and $\mathbb{Q}(\zeta_p)$ is a totally complex quadratic extension of $\mathbb{Q}\left(\zeta_p+\frac{1}{\zeta_p}\right)$.
 \end{lemma}
 \begin{proof}
  \begin{align*}
   \left[\mathbb{Q}(\zeta_p):\mathbb{Q}\left(\zeta_p+\frac{1}{\zeta_p}\right)\right]=2,\left[\mathbb{Q}(\zeta_p):\mathbb{Q}\right]=p-1
  \end{align*}
  imply
  \begin{align*}
   \left[\mathbb{Q}\left(\zeta_p+\frac{1}{\zeta_p}\right):\mathbb{Q}\right]=\frac{p-1}{2}.
  \end{align*}
  For $1\leq i\leq\frac{p-1}{2}$, define
  \begin{align*}
   \begin{array}{cccccc}
    \sigma_i\colon & \mathbb{Q}\left(\zeta_p+\frac{1}{\zeta_p}\right) & \hookrightarrow & \mathbb{Q}(\zeta_p) & \hookrightarrow & \mathbb{C}\\
    & z & \mapsto & z & & \\
    & & & \zeta_p & \mapsto & \zeta_p^{i} 
   \end{array}.
  \end{align*}
  These are all different and so they are the $\frac{p-1}{2}$ numbers of $\mathbb{Q}$-embeddings.\par
  $\zeta_p+\frac{1}{\zeta_p}$ maps to $\zeta_p^i+\frac{1}{\zeta_p^i}\in\mathbb{R}$ via $\sigma_i$, and so $\mathbb{Q}\left(\zeta_p+\frac{1}{\zeta_p}\right)$ is a totally real number field.\par
  $\mathbb{Q}(\zeta_p)\supset\mathbb{Q}(\zeta_p+\frac{1}{\zeta_p})$ is a quadratic extension and $\mathbb{Q}(\zeta_p)$ is totally complex and so the lemma holds. 
 \end{proof}
 \begin{lemma}\label{irreducibility}
  Let $F$ be a field.\par
  For a prime number $p$ and $a\in F^{\times}\backslash F^{\times p}$, $x^p-a\in F[x]$ is an irreducible polynomial.
 \end{lemma}
 \begin{proof}
  Let $F'$ be a splitting field of $x^p-a\in F[x]$.\par
  Assume $x^p-a\in F[x]$ is reducible and denote $x^p-a=f(x)g(x)\in F[x]$ with $f(x)=(x-z_1)\cdots(x-z_n),g(x)=(x-z_{n+1})\cdots(x-z_p)\in F[x]$ for $z_1, z_2,\ldots, z_p\in F'$ ($1\leq n<p$).\par
  Then $z_1\cdots z_n\in F^\times$ and so $(z_1\cdots z_n)^p=a^n\in F^{\times p}$ and so by $\mathrm{gcd}(n,p)=1$, $a\in F^{\times p}$, a contradiction.
 \end{proof}
 \begin{lemma}\label{cyclic extension1}
  For a prime number $p>3$ and an integer $1\leq k\leq\frac{p-1}{2}$, $\mathbb{Q}\left(\zeta_p,\sqrt[p]{\zeta_p^k+\frac{1}{\zeta_p^k}}\right)/\mathbb{Q}(\zeta_p)$ is a cyclic extension of degree $p$. 
 \end{lemma}
 \begin{proof}
  The extension $\mathbb{Q}\left(\zeta_p,\sqrt[p]{\zeta_p^k+\frac{1}{\zeta_p^k}}\right)/\mathbb{Q}(\zeta_p)$ is separable since $\mathbb{Q}$ has characteristic $0$.\par
  Any root of $x^p-\left(\zeta_p^k+\frac{1}{\zeta_p^k}\right)$ can be written as $\zeta_p^i\sqrt[p]{\zeta_p^k+\frac{1}{\zeta_p^k}}$ for some integer $i$ and it is contained in $\mathbb{Q}\left(\zeta_p,\sqrt[p]{\zeta_p^k+\frac{1}{\zeta_p^k}}\right)$ and so the extension $\mathbb{Q}\left(\zeta_p,\sqrt[p]{\zeta_p^k+\frac{1}{\zeta_p^k}}\right)/\mathbb{Q}(\zeta_p)$ is normal.
  For proving $\mathbb{Q}\left(\zeta_p,\sqrt[p]{\zeta_p^k+\frac{1}{\zeta_p^k}}\right)/\mathbb{Q}(\zeta_p)$ is a field extension of degree $p$, by Lemma \ref{irreducibility}, it is enough to prove $\zeta_p^k+\frac{1}{\zeta_p^k}\notin{\mathbb{Q}(\zeta_p)}^{\times p}$.\par
  Assume $\zeta_p^k+\frac{1}{\zeta_p^k}\in{\mathbb{Q}(\zeta_p)}^{\times p}$ and denote $\zeta_p^k+\frac{1}{\zeta_p^k}=\alpha^p$ for $\alpha\in{\mathbb{Q}(\zeta_p)}^{\times}$.
  Since $\zeta_p^k+\frac{1}{\zeta_p^k}$ is an algebraic integer, so is $\alpha$.
  The ring of integer of $\mathbb{Q}(\zeta_p)$ is $\mathbb{Z}[\zeta_p]$ (cf.\ \cite[Chapter 2.9]{Sam70}), and so $\alpha$ can be written as $\alpha=a_{p-2}\zeta_p^{p-2}+\cdots+a_1\zeta_p+a_0$ with $a_i\in\mathbb{Z}$.
  The equation
  \begin{equation}
   (a_{p-2}\zeta_p^{p-2}+\cdots+a_1\zeta_p+a_0)^p-(\zeta_p^{p-k}+\zeta_p^k)=0 \tag*{($\ast$)}
  \end{equation}
  holds and the left side is transformed to
  \begin{align*}
   b_{p-1}\zeta_p^{p-1}+\cdots+b_1\zeta_p+b_0
  \end{align*}
  with
  \begin{align*}
   &b_0\equiv a_{p-2}^p+\cdots+a_0^p\ (\mathrm{mod}\ p)\\
   &b_k\equiv -1\ (\mathrm{mod}\ p)\\
   &b_{p-k}\equiv -1\ (\mathrm{mod}\ p)\\
   &b_i\equiv0\ (\mathrm{mod}\ p)\quad(1\leq i\leq p-1, i\neq k,p-k)
  \end{align*}
  by calculating the combinatorial numbers.\par
  The minimal polynomial of $\zeta_p$ is $x^{p-1}+\cdots+x+1\in\mathbb{Z}[x]$ and so by combining with $p>3$, the equation ($\ast$) cannot hold, a contradiction.\par
  Thus, $\zeta_p^k+\frac{1}{\zeta_p^k}\notin{\mathbb{Q}(\zeta_p)}^{\times p}$ and so the extension $\mathbb{Q}\left(\zeta_p,\sqrt[p]{\zeta_p^k+\frac{1}{\zeta_p^k}}\right)/\mathbb{Q}(\zeta_p)$ is of degree $p$ and the Galois group is cyclic.
 \end{proof}

 \begin{lemma}\label{construction of anti-involution}
  Let $p$ be an odd prime number and denote $K=\mathbb{Q}(\zeta_p)$, $F=\mathbb{Q}\left(\zeta_p+\frac{1}{\zeta_p}\right)$.
  Let $L=\mathbb{Q}(\zeta_p,\sqrt[p]{a})$ be a cyclic extension of $\mathbb{Q}(\zeta_p)$ for $a\in\mathbb{Q}\left(\zeta_p+\frac{1}{\zeta_p}\right)$ with $[\mathbb{Q}(\zeta_p,\sqrt[p]{a}):\mathbb{Q}(\zeta_p)]=p$.
  Define the map
  \begin{align*}
   \begin{array}{cccc}
    \rho\colon & \mathbb{Q}(\zeta_p,\sqrt[p]{a}) & \rightarrow & \mathbb{Q}(\zeta_p,\sqrt[p]{a})\\
    & \zeta_p & \mapsto & \zeta_p \\
    & \sqrt[p]{a} & \mapsto & \zeta_p\sqrt[p]{a}
   \end{array}
  \end{align*}
  as $\mathrm{Gal}(L/K)$ is generated by $\rho$.
  Define $B=\oplus_{i=0}^{p-1} u^{i}L$ for a symbol $u$ with $u^p\in\mathbb{Q}(\zeta_p)^\times$ and multiplication on $B$ as $s\cdot u=u\cdot\rho(s)$ for $s\in L$.\par
  On this condition, there exists an anti-involution on $B$ of second kind.
  Moreover, by assuming that $B$ is a division ring, then there exists a positive anti-involution on $B$ of second kind.
 \end{lemma}
 \begin{proof}
  Define the map
  \begin{align*}
   \begin{array}{ccccc}
    \phi\colon & B & \rightarrow & B & \\
    & u^i\zeta_p^j\sqrt[p]{a^k} & \mapsto & u^i\zeta_p^{ik-j}\sqrt[p]{a^k} & (0\leq i,k\leq p-1,\ 0\leq j\leq p-2)
   \end{array}
  \end{align*}
  as a linear map over $\mathbb{Q}$.
  $\phi$ restricts to $K$ as
  \begin{align*}
   \begin{array}{ccccc}
    \phi|_K\colon & K & \rightarrow & K & \\
    & \zeta_p^j & \mapsto & \zeta_p^{-j} &  (0\leq j\leq p-1)
   \end{array}
  \end{align*}
  and $\phi$ fixes the field $F$ pointwise and so $\phi(a)=a$.
  Thus, by the equations
  \begin{align*}
   &\phi(1)=1,\\
   &\phi(\phi(u^i\zeta_p^j\sqrt[p]{a^k}))=\phi(u^i\zeta_p^{ik-j}\sqrt[p]{a^k})=u^i\zeta_p^j\sqrt[p]{a^k},
  \end{align*}
  \begin{align*}
   \phi(u^i\zeta_p^j\sqrt[p]{a^k}\cdot u^{i'}\zeta_p^{j'}\sqrt[p]{a^{k'}})&=\phi(u^i u^{i'}\cdot\rho^{i'}(\zeta_p^j\sqrt[p]{a^k})\zeta_p^{j'}\sqrt[p]{a^{k'}})\\
                                                                          &=\phi(u^{i+i'}\zeta_p^{i'k+j}\sqrt[p]{a^k}\zeta_p^{j'}\sqrt[p]{a^{k'}})\\
                                                                          &=\phi(u^{i+i'}\zeta_p^{i'k+j+j'}\sqrt[p]{a^{k+k'}})\\
                                                                          &=u^{i+i'}\zeta_p^{(i+i')(k+k')-(i'k+j+j')}\sqrt[p]{a^{k+k'}}
  \end{align*}
  and
  \begin{align*}
   \phi(u^{i'}\zeta_p^{j'}\sqrt[p]{a^{k'}})\phi(u^i\zeta_p^j\sqrt[p]{a^k})&=u^{i'}\zeta_p^{i'k'-j'}\sqrt[p]{a^{k'}}\cdot u^i\zeta_p^{ik-j}\sqrt[p]{a^k}\\
                                                                          &=u^{i'}u^i\cdot\rho^{i}(\zeta_p^{i'k'-j'}\sqrt[p]{a^{k'}})\zeta_p^{ik-j}\sqrt[p]{a^k}\\
                                                                          &=u^{i+i'}\zeta_p^{i'k'-j'+ik'}\sqrt[p]{a^{k'}}\zeta_p^{ik-j}\sqrt[p]{a^k}\\
                                                                          &=u^{i+i'}\zeta_p^{i'k'-j'+ik'+ik-j}\sqrt[p]{a^{k+k'}},
  \end{align*}
  $\phi$ is an anti-involution on $B$.\par
  As above, $\phi|_K$ fixes only the elements in $F$ pointwise and so $\phi$ is of second kind (cf.\ Definition \ref{second kind}).
  Moreover if $B$ is a division ring, by Theorem \ref{existence of positive anti-involution}, there exists a positive anti-involution on $B$ of second kind.
 \end{proof}
 \begin{lemma}\label{multiplicity}
  Let $L/K$ be a Galois extension of number fields with $[L:K]=n$ and let $\mathfrak{q}\subset\mathcal{O}_L$ be a prime ideal and take $\alpha\in L^{\times}\backslash U_L$ ($U_L$ is the set of the invertible elements of $\mathcal{O}_L$).\par
  Assume $\sigma(\mathfrak{q})=\mathfrak{q}$ for all $\sigma\in\mathrm{Gal}(L/K)$.
  Then, the multiplicity of $\mathfrak{q}$ of the prime ideal factorization of the ideal generated by $\mathrm{N}_{L/K}(\alpha)$ in $\mathcal{O}_L$ is a multiple of $n$.
 \end{lemma}
 \begin{proof}
  Write the prime ideal factorization of $\alpha\mathcal{O}_L$ as
  \begin{align*}
   \alpha\mathcal{O}_L={\mathfrak{q}_1}^{e(\mathfrak{q}_1)}\cdots{\mathfrak{q}_g}^{e(\mathfrak{q}_g)}\quad(\mathfrak{q}_i\text{ are all different})
  \end{align*}
  and then
  \begin{align*}
   \sigma(\alpha)\mathcal{O}_L=\sigma(\mathfrak{q}_1)^{e(\mathfrak{q}_1)}\cdots\sigma(\mathfrak{q}_g)^{e(\mathfrak{q}_g)}
  \end{align*}
  for each $\sigma\in\mathrm{Gal}(L/K)$.\par
  Also, the next equation holds.
  \begin{align*}
   \left(\mathrm{N}_{L/K}(\alpha)\right)\mathcal{O}_L=\left(\prod_{\sigma\in\mathrm{Gal}(L/K)} \sigma(\alpha)\right) \mathcal{O}_L=\prod_{\sigma\in\mathrm{Gal}(L/K)} \prod_{i=1}^g \sigma(\mathfrak{q}_i)^{e(\mathfrak{q}_i)}
  \end{align*}\par
  By the assumption that $\sigma(\mathfrak{q})=\mathfrak{q}$ for all $\sigma\in\mathrm{Gal}(L/K)$, $\mathfrak{q}$ appears in the factorization of $\left(\mathrm{N}_{L/K}(\alpha)\right)\mathcal{O}_L$ if and only if it appears in the factorization of $\alpha\mathcal{O}_L$.\par
  Moreover, if $\mathfrak{q}$ appears in $\alpha\mathcal{O}_L$ and put $\mathfrak{q}_1=\mathfrak{q}$, then
  \begin{align*}
   \left(\mathrm{N}_{L/K}(\alpha)\right)\mathcal{O}_L=\mathfrak{q}^{e(\mathfrak{q})\cdot n}\prod_{\sigma\in\mathrm{Gal}(L/K)} \prod_{i=2}^g \sigma(\mathfrak{q}_i)^{e(\mathfrak{q_i})},
  \end{align*}
  and so the multiplicity of $\mathfrak{q}$ is a multiple of $n$.
 \end{proof}
 By combining these lemmas, Theorem \ref{automorphism} can be proved.
 \begin{proof}[Proof of Theorem \ref{automorphism}]
  Denote $\gamma=\zeta_p^k+\frac{1}{\zeta_p^k}$.
  Consider the sequence of Galois extensions $\mathbb{Q}(\zeta_p,\sqrt[p]{\gamma})\supset\mathbb{Q}(\zeta_p)\supset\mathbb{Q}$ of fields and denote $K=\mathbb{Q}(\zeta_p)$, $L=\mathbb{Q}(\zeta_p,\sqrt[p]{\gamma})$.
  Now $L/K$ is a cyclic extension of degree $p$ by Lemma \ref{cyclic extension2}.\par
  Also denote $K_0=\mathbb{Q}\left(\zeta_p+\frac{1}{\zeta_p}\right)$ and $G=\mathrm{Gal}(L/K)$.
  By Lemma \ref{CM-field}, $K$ is a totally complex quadratic extension of the totally real number field $K_0$ and so, $K$ satisfies the conditions of CM-field.\par
  Define
  \begin{align*}
   \begin{array}{cccc}
    \rho\colon & L & \rightarrow & L \\
    & \zeta_p & \mapsto & \zeta_p \\ 
    & \sqrt[p]{\alpha} & \mapsto & \sqrt[p]{\alpha}\zeta_p
   \end{array}.
  \end{align*}
  By using Theorem \ref{Cebotarev's density theorem} for the Galois extension $L/K$ and $\rho$, there exist infinitely many pairs $(\mathfrak{p},\mathfrak{q})$ of a prime ideal $\mathfrak{p}\subset\mathcal{O}_K$ and a prime ideal $\mathfrak{q}\subset\mathcal{O}_L$ over $\mathfrak{p}$ which satisfy the next conditions.
  \begin{enumerate}
  \renewcommand{\labelenumi}{\arabic{enumi}.}
   \item $\mathfrak{p}$ is unramified in $L$.
   \item The Frobenius automorphism $\left(\frac{K/F}{\mathfrak{q}}\right)$ is $\rho$.
  \end{enumerate}
  Fix a pair $(\mathfrak{p},\mathfrak{q})$ and define the prime number $q$ as $q\mathbb{Z}=\mathfrak{p}\cap\mathbb{Z}$.\par
  By Hilbert's ramification theory, it can be written as
  \begin{align*}
   q\mathcal{O}_K=\prod_{i=1}^g {\mathfrak{p}_i}^e\quad(\mathfrak{p}_i\text{ are all different}),\\
   {\mathfrak{p}_i}\mathcal{O}_L=\prod_{j=1}^{g_i} {\mathfrak{q}_{ij}}^{e_i}\quad(\mathfrak{q}_{ij}\text{ are all different}).
  \end{align*}
  Then, 
  \begin{align*}
   q\mathcal{O}_L=\prod_{i=1}^g \prod_{j=1}^{g_i} {\mathfrak{q}_{ij}}^{e\cdot e_i}
  \end{align*}
  and denote $f=[\mathcal{O}_K/\mathfrak{p}_i:\mathbb{Z}/q\mathbb{Z}]$, $f_i=[\mathcal{O}_L/\mathfrak{q}_{ij}:\mathcal{O}_K/\mathfrak{p}_i]$ with $efg=p-1$, $e_if_ig_i=p$ by the ramification theory again.\par
  Denote $\mathfrak{p}_1=\mathfrak{p}$ and $\mathfrak{q}_{11}=\mathfrak{q}$ in the above prime ideal factorizations.
  The condition 1 implies $e_1=1$ and by $e\leq p-1$, $ee_1$ is not a multiple of $p$.
  The condition 2 implies $\rho(\mathfrak{q})=\mathfrak{q}$ and so the decomposition group $\mathcal{D}_\mathfrak{q}$ is equal to $G$.
  The order of the decomposition group $\mathcal{D}_\mathfrak{q}$ is $e_1f_1$ and so $f_1=p$, and this implies $g_1=1$.
  Therefore the pair $(q,\mathfrak{q})$ satisfies the next conditions.
  \begin{itemize}
   \item $\sigma(\mathfrak{q})=\mathfrak{q}$ for all $\sigma\in G$
   \item the multiplicity of $\mathfrak{q}$ of the prime ideal factorization of the ideal $q\mathcal{O}_L$ is not a multiple of $p$.
  \end{itemize}
  Since $q$ is not an invertible element in $\mathcal{O}_L$, $q\not\in\mathrm{N}_{L/K}(U_L)$.
  Thus, by Lemma \ref{multiplicity}, $q$ cannot be written as $q=\mathrm{N}_{L/K}(\alpha)$ for any $\alpha\in L^{\times}$.\par
  Moreover, $q^p=\mathrm{N}_{L/K}(q)\in\mathrm{N}_{L/K}(L^{\times})$ and since $p$ is a prime number, the order of $q$ in $K^{\times}/\mathrm{N}_{L/K}(L^{\times})$ is exactly $p$.\par
  By applying Proposition \ref{construct division algebra}, there is a central simple division algebra $B$ of the form $B=\oplus_{i=0}^{p-1} u^{i}L$ with a symbol $u$ which satisfies $u^p=q$.
  Now the division algebra $B$ satisfies all the conditons in Theorem \ref{automorphism}.
 \end{proof}
 Now, $\zeta_p^k+\frac{1}{\zeta_p^k}$ is an algebraic integer and also $\left(\zeta_p^k+\frac{1}{\zeta_p^k}\right)^{-1}$ is an algebraic integer by Theorem \ref{constant term of minimal polynomial}.
 Thus, $\sqrt[p]{\zeta_p^k+\frac{1}{\zeta_p^k}}$ and $\sqrt[p]{\zeta_p^k+\frac{1}{\zeta_p^k}}^{-1}$ are both algebraic integers, and this implies the next result.
 \begin{theorem}\label{minimum value of Delta}
  There does not exist the minimum value of $\Delta$.
 \end{theorem}
 \begin{proof}
  For a prime number $p>3$ and an integer $1\leq k\leq\frac{p-1}{2}$, there is a finite-dimensional central simple division algebra over the CM-field $\mathbb{Q}(\zeta_p)$ which contains $\sqrt[p]{\zeta_p^k+\frac{1}{\zeta_p^k}}$ by Theorem \ref{automorphism}.
  The constructed division algebra $B=\oplus_{i=0}^{p-1} u^{i}L$ in Theorem \ref{automorphism} has an order $\mathcal{O}=\oplus_{i=0}^{p-1} u^{i}\mathcal{O}_L$ by Remark \ref{remark about the order} and $\mathcal{O}$ contains $\sqrt[p]{\zeta_p^k+\frac{1}{\zeta_p^k}}$ and $\sqrt[p]{\zeta_p^k+\frac{1}{\zeta_p^k}}^{-1}$ since these are integral elements.
  By combining this with Lemma \ref{construction of anti-involution} and Lemma \ref{Construction4} for $d=p,e_0=\frac{p-1}{2}$ and $m=1$, there is an automorphism of a $\frac{p^2(p-1)}{2}$-dimensional simple abelian variety which corresponds to $\sqrt[p]{\zeta_p^k+\frac{1}{\zeta_p^k}}$.\par
  The minimal polynomial of $\sqrt[p]{\zeta_p^k+\frac{1}{\zeta_p^k}}$ over $\mathbb{Q}$ is
  \begin{align*}
   \prod_{k=1}^{\frac{p-1}{2}} \left(x^p-\left(\zeta_p^k+\frac{1}{\zeta_p^k}\right)\right)
  \end{align*}
  and so its first dynamical degree is $\sqrt[p]{\abs[\bigg]{\zeta_p^{\frac{p-1}{2}}+\frac{1}{\zeta_p^{\frac{p-1}{2}}}}^2}$ as in Section \ref{Calculation}.\par
  Now, $1<\sqrt[p]{\abs[\bigg]{\zeta_p^{\frac{p-1}{2}}+\frac{1}{\zeta_p^{\frac{p-1}{2}}}}^2}<\sqrt[p]{4}$ converges to $1$ as $p\to\infty$. 
  Thus, for any $\epsilon>0$, there is a first dynamical degree $\lambda$ of an automorphism of a simple abelian variety such that $1<\lambda<1+\epsilon$.
 \end{proof}
 By removing the automorphic condition, the next theorem holds.
 \begin{theorem}\label{analogous theorem}
  For any $\epsilon>0$, there exists a simple abelian variety $X$ and an endomorphism $f:X\rightarrow X$ which is not an automorphism such that $1<\lambda_1(f)<1+\epsilon$.
 \end{theorem}
 The thereom is proved by using the next lemma instead of Lemma \ref{cyclic extension1}.
 \begin{lemma}\label{cyclic extension2}
  For a positive integer $a$ and a prime number $p$ such that $\sqrt[p]{a}\notin\mathbb{Z}$, $\mathbb{Q}(\zeta_p,\sqrt[p]{a})/\mathbb{Q}(\zeta_p)$ is a cyclic extension of degree $p$. 
 \end{lemma}
 \begin{proof}
  The extension $\mathbb{Q}(\zeta_p,\sqrt[p]{a})/\mathbb{Q}(\zeta_p)$ is separable and normal as in the proof of Lemma \ref{cyclic extension1}.\par
  Assume $a\in{\mathbb{Q}(\zeta_p)}^{\times p}$ and denote $a=\alpha^p$ for $\alpha\in{\mathbb{Q}(\zeta_p)}^{\times}$.
  By the irreducibility of $x^p-a\in\mathbb{Z}[x]$ from $\sqrt[p]{a}\notin\mathbb{Z}$, $[\mathbb{Q}(\alpha):\mathbb{Q}]=p$ and so this is in contradiction with $\mathbb{Q}(\zeta_p)\supset\mathbb{Q}(\alpha)\supset\mathbb{Q}$.\par
  Thus, $x^p-a$ is irreducible over $\mathbb{Q}(\zeta_p)$ by Lemma \ref{irreducibility} and so the extension $\mathbb{Q}(\zeta_p,\sqrt[p]{a})/\mathbb{Q}(\zeta_p)$ is of degree $p$.\par
  Therefore $\mathbb{Q}(\zeta_p,\sqrt[p]{a})/\mathbb{Q}(\zeta_p)$ is a Galois extension and the Galois group $\mathrm{Gal}\left(\mathbb{Q}(\zeta_p,\sqrt[p]{a})/\mathbb{Q}(\zeta_p)\right)$ has order $p$ and this is the cyclic group generated by $\sqrt[p]{a}\mapsto\zeta_p\sqrt[p]{a}$.
 \end{proof}
 \begin{theorem}\label{endomorphism not automorphism}
  For an odd prime number $p$, a positive integer $a$ with $\sqrt[p]{a}\not\in\mathbb{Z}$, there are a CM-field $K$, a central simple division algebra $B$ over $K$ of the form $B=\oplus_{i=0}^{p-1} u^{i}L$ with a symbol $u$, a cyclic extension $L/K$ of degree $p$, and $[B:K]=p^2$ such that $\sqrt[p]{a}\in L$.
 \end{theorem}
 \begin{proof}
  Denote $K=\mathbb{Q}(\zeta_p)$, $L=\mathbb{Q}\left(\zeta_p, \sqrt[p]{a}\right)$ as in the proof of Theorem \ref{automorphism}.
  $L\supset K\supset\mathbb{Q}$ is a sequence of Galois extensions and $L/K$ is a cyclic extension of degree $p$ by Lemma \ref{cyclic extension2}.\par
  As in the proof of Theorem \ref{automorphism}, by applying Theorem \ref{Cebotarev's density theorem} for the Galois extension $L/K$, there exists a pair $(q,\mathfrak{q})$ of a prime number $q$ and a prime ideal $\mathfrak{q}\subset\mathcal{O}_L$ such that 
  \begin{itemize}
   \item $\sigma(\mathfrak{q})=\mathfrak{q}$ for all $\sigma\in\mathrm{Gal}(L/K)$
   \item The multiplicity of $\mathfrak{q}$ of the prime ideal factorization of the ideal $q\mathcal{O}_L$ is not a multiple of $p$.
  \end{itemize}
  Thus, as in the proof of Theorem \ref{automorphism}, by Lemma \ref{multiplicity}, the order of $q$ in $K^{\times}/\mathrm{N}_{L/K}(L^{\times})$ is exactly $p$.\par
  By Proposition \ref{construct division algebra}, $B=\oplus_{i=0}^{p-1} u^{i}L$ with a symbol $u$ which satisfies $u^p=q$ is a central simple division algebra over $K$ with $[B:K]=p^2$ and satisfies all the conditons in Theorem \ref{endomorphism not automorphism}.
 \end{proof}
 \begin{proof}[Proof of Theorem \ref{analogous theorem}]
  For an integer $a$ and an odd prime number $p$ with $\sqrt[p]{a}\not\in\mathbb{Z}$, the central simple division algebra over the CM-field $\mathbb{Q}(\zeta_p)$ constructed in Theorem \ref{endomorphism not automorphism} has an order which contains $\sqrt[p]{a}$ by Remark \ref{remark about the order}.
  Also by Lemma \ref{construction of anti-involution} and Lemma \ref{Construction4} for $d=p,e_0=\frac{p-1}{2}$ and $m=1$, there is an endomorphism of a $\frac{p^2(p-1)}{2}$-dimensional simple abelian variety which corresponds to $\sqrt[p]{a}$ and its first dynamical degree is $\sqrt[p]{a^2}$ as in Section \ref{Calculation}.\par
  By substituting $a=2$, $\sqrt[p]{a}\not\in\mathbb{Z}$ holds and $\sqrt[p]{a^2}$ converges to $1$ as $p\to\infty$.
  Thus, the first dynamical degrees of this type of endomorphisms are not $1$ and have no mimimum value.\par
  Also, an endomorphism of this type cannot be an automorphism because $\frac{1}{\sqrt[p]{a}}$ is not an algebraic integer and so the proof is concluded.
 \end{proof}

\renewcommand{\refname}{References}

\end{document}